\documentclass[a4paper,11pt]{amsart}
\pdfoutput=1 

\usepackage[utf8]{inputenc}
\usepackage[T1]{fontenc}
\usepackage{fullpage}

\usepackage{microtype}
\usepackage{amsfonts,amssymb}
\usepackage{amsthm,amsmath}
\usepackage{mathtools}
\usepackage{textcomp}
\usepackage{enumitem}
\usepackage[sort]{cite}
\newlist{paraenum}{enumerate}{1}
\setlist[paraenum]{wide, label=(\arabic*)}
\newlist{inparaenum}{enumerate*}{1}
\setlist[inparaenum]{label=(\arabic*)}
\setlist[itemize]{leftmargin=4em,rightmargin=4.5em,label=---}
\setlist[enumerate]{leftmargin=4em,rightmargin=4.5em}
\setlist[description]{leftmargin=4em,labelindent=4em,rightmargin=4.5em}

\usepackage{graphicx}
\usepackage{url}
\usepackage{hyperref}
\usepackage[sort]{cleveref}
\usepackage{xcolor}

\definecolor{darkred}{RGB}{160,0,0}
\definecolor{darkblue}{RGB}{0,0,160}
\hypersetup{
  colorlinks,
  citecolor=darkblue,
  filecolor=black,
  linkcolor=darkblue,
  urlcolor=darkblue
}

\makeatletter
\@namedef{subjclassname@2020}{
  \textup{2020} Mathematics Subject Classification}
\makeatother
\makeatletter
\newcommand{\kbldelim}{(}
\newcommand{\kbrdelim}{)}

\newcommand{\kbrowstyle}{\scriptstyle}
\newcommand{\kbcolstyle}{\scriptstyle}

\newlength{\kbcolsep}
\newlength{\kbrowsep}

\newif\ifkbalignright

\newlength{\br@kwd}
\newlength{\k@bordht}

\newcommand{\kbordermatrix}[1]{
\begingroup
	\setbox0=\hbox{$\left\kbldelim\right.$}
	\setlength{\br@kwd}{\wd0}
	\setbox\@arstrutbox\hbox{\vrule
		\@height\arraystretch\ht\strutbox
		\@depth\arraystretch\dp\strutbox
		\@width\z@}
	\setlength{\k@bordht}{\kbrowsep}
	\addtolength{\k@bordht}{\ht\@arstrutbox}
	\addtolength{\k@bordht}{\dp\@arstrutbox}
	\m@th
	\def\@kbrowstyle{\kbrowstyle}
	\setbox0=\vbox{
		\def\cr{\crcr\noalign{\kern\kbrowsep
			\global\let\cr=\endline
			\global\let\@kbrowstyle=\relax}}
		\let\\\@arraycr
		\lineskip\z@skip
		\baselineskip\z@skip
		\dimen0\kbcolsep \advance\dimen0\br@kwd
		\ialign{\tabskip\dimen0
			\kern\arraycolsep\hfil\@arstrut$\kbcolstyle ##$\hfil\kern\arraycolsep&
			\tabskip\z@skip
			\kern\arraycolsep\hfil$\@kbrowstyle ##$\ifkbalignright\relax\else\hfil\fi\kern\arraycolsep&&
			\kern\arraycolsep\hfil$\@kbrowstyle ##$\ifkbalignright\relax\else\hfil\fi\kern\arraycolsep\crcr
			#1\crcr}
	}
	\setbox2=\vbox{\unvcopy0 \global\setbox5=\lastbox}
	\loop
		\setbox2=\hbox{\unhbox5 \unskip \global\setbox3=\lastbox}
		\ifhbox3
			\global\setbox5=\box2
			\global\setbox1=\box3
	\repeat
	\setbox2=\hbox{$\kern\wd1\kern\kbcolsep\kern-\arraycolsep
		\left\kbldelim
		\kern-\wd1\kern-\kbcolsep\kern-\br@kwd
	\vcenter{\kern-\k@bordht\vbox{\unvbox0}}
	\right\kbrdelim$}
	\null\vbox{\kern\k@bordht\box2}
	\endgroup
}
\makeatother
\usepackage{bm}
\makeatletter
\g@addto@macro\bfseries{\boldmath}
\makeatother

\usepackage{stmaryrd}
\newcommand{\lsem}{\llbracket}
\newcommand{\rsem}{\rrbracket}

\newcommand{\A}{\mathcal{A}}
\newcommand{\G}[1]{\mathcal{#1}}

\newcommand{\BB}[1]{\mathbb{#1}}
\newcommand{\Sym}{S}
\newcommand{\HypOct}{B}
\newcommand{\SL}{\mathrm{SL}}
\newcommand{\GF}{\mathrm{GF}}

\DeclareMathOperator{\sgn}{sgn}
\DeclareMathOperator{\Adj}{adj}
\newcommand{\CI}{\mathrel{\text{$\perp\mkern-10mu\perp$}}}
\newcommand{\CIS}[1]{\lsem#1\rsem}

\newcommand{\eps}{\varepsilon}
\newcommand{\comp}{\mathsf{c}}
\newcommand{\ol}[1]{\overline{#1}}
\theoremstyle{definition}
\newtheorem{theorem}{Theorem}[section]
\newtheorem{mainthm}[theorem]{Theorem}
\newtheorem{lemma}[theorem]{Lemma}
\newtheorem{proposition}[theorem]{Proposition}
\newtheorem{corollary}[theorem]{Corollary}
\newtheorem{remark}[theorem]{Remark}

\newtheorem*{fact*}{Fact}
\newtheorem{example}[theorem]{Example}
\newtheorem{definition}[theorem]{Definition}

\newtheorem{conjecture}[theorem]{Conjecture}

\newtheorem{question}[theorem]{Question}

\usepackage{environ}
\NewEnviron{TodoList}[1][TODO]{
  \color{darkred}
  \textbf{#1:}
  \begin{itemize}
  \BODY
  \end{itemize}
}

\title[Gaussoids are two-antecedental Gaussians]%
{Gaussoids are two-antecedental approximations \\
of Gaussian conditional independence structures}
\author{Tobias Boege}
\address{Tobias Boege, Max-Planck Institute for Mathematics in the Sciences Leipzig, Germany}
\email{post@taboege.de}
\date{\today}

\begin{document}

\begin{abstract}
The gaussoid axioms are conditional independence inference rules which
characterize regular Gaussian CI structures over a three-element ground
set. It is known that no finite set of inference rules completely describes
regular Gaussian CI as the ground set grows. In this article we show that
the gaussoid axioms logically imply every inference rule of at most two
antecedents which is valid for regular Gaussians over any ground set.
The proof is accomplished by exhibiting for each inclusion-minimal gaussoid
extension of at most two CI statements a regular Gaussian realization.
Moreover we prove that all those gaussoids have rational positive-definite
realizations inside every $\eps$-ball around the identity matrix.
For the proof we introduce the concept of algebraic Gaussians over arbitrary
fields and of positive Gaussians over ordered fields and obtain the same
two-antecedental completeness of the gaussoid axioms for algebraic and
positive Gaussians over all fields of characteristic zero as a byproduct.
\end{abstract}

\maketitle

\section{Introduction}
\label{sec:Intro}

Conditional independence (CI) is a basic notion in the probabilistic
approach to reasoning under uncertainty. CI~constraints prescribe statements
of the form ``given that the value of the factor $C$ is known, the value of
$A$ is \emph{irrelevant} to the value of $B$'', or ``having observed $C$,
the outcome of $A$ gives no further information on the outcome of $B$'' on
a joint probability distribution of random variables $A, B, C, \dots$.
Such statements allow the incorporation of expert knowledge on independences
of the real-world phenomena into a statistical model.
The importance of conditional independence in statistical theory has been
described in a seminal paper by Dawid~\cite{DawidStatTheory}.
Pearl~\cite[Chapter~3]{PearlBook} further emphasizes the usefulness of CI
as a qualitative, as opposed to numeric, measure of independence in
artificial intelligence, in that reasoning on CI can be performed logically
instead of by summing over and dividing given probabilities.
The~advantage of using CI~inference in the processing and deduction of
further (in)dependence statements is that every conclusion which is drawn
from a given set of CI~assumptions is sound and that their derivation
happens in discrete steps, each of which is verifiable by humans based
on agreed-upon deduction rules.

The fundamental laws of conditional independence identified by Dawid
became the definition of \emph{semigraphoids}. The semigraphoid properties
are universal deduction rules which are valid for all probability
distributions~\cite[Lemma~2.1]{Studeny}. Thus they may immediately be
applied to any set of statements about the conditional independences
among random variables to derive additional knowledge about the independence
structure of the stochastic system.
The historical reference on the topic, after Dawid, is~\cite{PearlPaz},
where the special case of \emph{graphoids} is introduced and first named
in the context of graphical models. See also the historical overview in
\cite{GMStudeny}.
By~deduction or inference rules, we always mean \emph{CI~inference formulas}
(or \emph{inference forms})
\[
  \label{eq:Inf} \tag{$\Rightarrow$}
  \bigwedge_s a_s \;\Rightarrow\; \bigvee_t c_t
\]
demanding that if all of the antecedent statements~$a_s$ are satisfied,
then at least one of the consequent statements~$c_t$ is satisfied.
Which inference formulas are true, i.e., \emph{inference rules}, depends
on assumptions about the type of probability distribution. Distributions
with only binary random variables may satisfy more inference rules than
arbitrary discrete distributions, which may satisfy yet other rules than
the continuous Gaussian distributions.
For this reason, the structures of conditional independence for restricted
classes of distributions have become objects of mathematical study in their
own right. Given a class of distributions, such as Gaussians, a~central task
is to solve the \emph{inference problem}: to decide when an inference
form~\eqref{eq:Inf} is valid for all distributions in the class, and then
to list all of them for use in expert systems. See~\cite{GeigerPearl} for
early work and the recent~survey~\cite{InformationDecision}.

For~a detailed introduction to probabilistic representations of CI~structures,
we refer to~\cite{Studeny}. The~recent book~\cite{SullivantBook} contains an
algebraically-minded introduction to CI of discrete and Gaussian distributions
in Chapter~4. Chapter~13 of the same book discusses undirected and directed
acyclic graphical models. These have been of interest since at least the
works of Lauritzen~\cite{LauritzenBook} and Pearl~\cite{PearlBook} and have
found many uses in practice. These types of graphs are special cases of
both, discrete as well as Gaussian, CI~models; see~\cite[Section~13.2]{SullivantBook}.
An~overview and unification of the CI~inference calculi of various other sorts
of graphical models can be found~in~\cite{UnifyingMarkov}.

The research into CI~structures of discrete random variables in the late
1980s was in part driven by the conjecture of Pearl and Paz that to every
semi\-graphoid a matching vector of discrete random variables may be found
whose CI~structure is exactly that semigraphoid~\cite[p.~88]{PearlBook}.
An alternative formulation of this conjecture, in logical terminology, is
that the semigraphoid axioms are \emph{complete} for the theory of discrete
CI~structures.
This conjecture was refuted by Studený who exhibited an infinite family of valid
inference rules for discrete CI which are not implied by any finite set of
valid inference rules~\cite{StudenyNonfinite}. This not only proves that the
semigraphoid axioms are not complete, but that no finite list of axioms can
be complete for discrete CI.
Studený's inference rules naturally require a growing number of random
variables but they also use more and more antecedents. A~finite complete
list of valid inference rules being impossible to obtain, the conjecture
was revised to state that the semigraphoid axioms are complete for only
those inference rules of discrete CI which have at most two antecedents,
as the semigraphoid axioms themselves do. This was in turn resolved
positively by Studený~\cite{StudenyTwoAntecedental}. In this sense,
the semigraphoid axioms are the most fundamental laws of CI on discrete
random vectors: they logically imply all the non-trivial valid inference
rules with the lowest number of antecedents, or premises --- those which
one would expect to be most acceptable in human reasoning.

The same story unfolds for multivariate regular Gaussian distributions,
with more or less 15 years delay. In the case of Gaussian distributions,
CI~structures strike a fine balance between varied and highly structured.
For example, the inference mechanisms of Markov and Bayesian networks in
graphical modeling are special cases of Gaussian CI~models, but the general
theory is not more complicated than real algebraic geometry
(see~\Cref{ImplicationAlgebraic}).
The role of the semigraphoid axioms as the most essential inference rules
is here played by the \emph{gaussoid axioms}. These, likewise two-antecedental,
inference rules have been found by Matúš in~\cite{MatusGaussian} and the
term \emph{gaussoid} was coined in the work~\cite{LnenickaMatus} by Lněnička
and Matúš who classified the CI~structures which arise from Gaussians on
three and four random variables. Realizability in the three-variate case
is characterized by the gaussoid axioms. On four random variables, higher
inference rules arise, but none with only two antecedents. Around the same
time, Sullivant~\cite{SullivantNonfinite}, and independently Šimeček~\cite{SimecekNonfinite}
in a different framework, showed that no finite list of inference rules
suffices to deduce all valid inference rules for $n$-variate Gaussian
distributions, as $n$ grows.
In~this article we prove the Gaussian analogue of Studený's two-antecedental
completeness result in~\Cref{TwoAntecedental}: every valid inference rule
for Gaussians on any number of random variables but with at most two antecedents
can be deduced from the gaussoid axioms. Hence, the gaussoid axioms are the
most essential deduction rules for CI~inference on Gaussian random variables.

This work is structured as follows: \Cref{sec:Prelim} gives an introduction
to Gaussian conditional independence structures and their symmetries.
\Cref{sec:MainResult} gives a precise statement of the main technical result
\Cref{MainThm} and derives the titular two-antecedental completeness result
\Cref{TwoAntecedental} from it. At the end of \Cref{sec:MainResult}, a more
detailed summary of the proof strategy is given. The subsequent sections
then form the technical part of the paper which execute the strategy:
\Cref{sec:Group} discusses an algebraic relaxation of Gaussian distributions
together with their symmetries. \Cref{sec:Epsilon} shows how Gaussian
realizability results can be recovered from this relaxation. The proof of
\Cref{MainThm} is presented in a series of lemmas in \Cref{sec:Proof}.
\Cref{sec:Remarks} contains further examples and discussion of
future work.

\section{Preliminaries}
\label{sec:Prelim}

We fix a finite ground set $N$ of size $n$ which labels a vector
$\xi = (\xi_i)_{i \in N}$ of random variables. Suppose that this vector
follows a multivariate Gaussian distribution with mean vector~$\mu$
and positive-definite covariance matrix~$\Sigma$. Its density with
respect to the Lebesgue measure on $\BB R^N$ is given by
\[
  x \;\mapsto\; \frac1{\sqrt{(2\pi)^n \det \Sigma}}
    \exp\left(-\frac12 (x-\mu)^T \Sigma^{-1} (x-\mu) \right).
\]
For distinct $i, j \in N$ and a disjoint subset $K \subseteq N$, the
\emph{conditional independence statement} $\xi_i \CI \xi_j \mid \xi_K$
asserts, informally, that whenever the outcome of the subvector $\xi_K
= (\xi_k)_{k \in K}$ is known, the outcomes of $\xi_i$ and $\xi_j$ are
stochastically independent --- learning one value provides no
additional information on the other. The CI~statement above is
abbreviated by the symbol $(ij|K)$.
In dealing with elements and subsets of the ground set~$N$, we adopt
the following notational conventions:
\begin{inparaenum}
\item an element $i \in N$ may be written in place of the singleton
subset $\{i\} \subseteq N$, and
\item juxtaposition $KL$ of subsets of $N$ abbreviates set union.
\end{inparaenum}
In particular $ijK$ stands for $\{i\} \cup \{j\} \cup K \subseteq N$
and $ij = ji$ holds. This does \emph{not} make the CI~statement
$(ij|K)$ ambiguous because conditional independence of random vectors
is symmetric in $i$~and~$j$. Let $\A_N := \{ (ij|K) : ij \in \binom{N}{2},
K \subseteq N \setminus ij \}$ denote the set of all CI~statements on~$N$.

For a Gaussian vector $\xi$ with positive-definite covariance matrix
$\Sigma$, conditional independence has an algebraic characterization
(see~\cite[Proposition~4.1.9]{SullivantBook}):
\begin{align*}
  \tag{$\CI$} \label{eq:CIdef}
  \text{$(ij|K)$ holds for $\xi$} \;\Leftrightarrow\;
  \det \Sigma_{iK,jK} = 0,
\end{align*}
where $\Sigma_{iK,jK}$ is the submatrix of $\Sigma$ with rows $iK$
and columns $jK$. A submatrix of the form $\Sigma_{K,K}$ with
equal row and column sets is \emph{principal}. A symmetric matrix~$\Sigma$
is positive-definite if and only if all principal minors $\det \Sigma_{K,K}$,
$K \subseteq N$, are positive.
In the submatrix $\Sigma_{iK,jK}$, which is decisive for the CI~statement
$(ij|K)$, the row and column sets $iK$ and~$jK$ differ by only one element
each. They are not principal but \emph{almost-principal} submatrices.
Hence, the true conditional independence statements of a given Gaussian
distribution can be determined by evaluating the almost-principal minors
of the covariance matrix. The mean $\mu$ plays no role and may be assumed
to be~zero. The \emph{CI~structure} of $\Sigma$ is the set of all valid
CI~statements:
\[
  \CIS{\Sigma} := \{ (ij|K) \in \A_N : \det \Sigma_{iK,jK} = 0 \}.
\]

\begin{example}
Let $(\xi_1, \xi_2, \xi_3)$ follow a Gaussian distribution with zero mean
and covariance matrix
\[
  \Sigma = \begin{pmatrix}
    2 & 0 & 1 \\
    0 & 2 & 2 \\
    1 & 2 & 3
  \end{pmatrix}.
\]
There are six possible CI~statements over the ground set $N = 123 = \{1,2,3\}$:
\[
  \A_N = \{ (12|), (12|3), (13|), (13|2), (23|), (23|1) \}.
\]
The three marginal CI~statements $(12|)$, $(13|)$ and $(23|)$ have empty
conditioning sets. Their corresponding almost-principal minors are
off-diagonal entries of $\Sigma$. The other three are $2 \times 2$
determinants:
\begin{equation*}
  \begin{array}{c}
    \det \Sigma_{12|}  = 0, \\
    \det \Sigma_{12|3} = \det \begin{psmallmatrix}
      0 & 1 \\ 2 & 3
    \end{psmallmatrix} = -1,
  \end{array} \quad
  \begin{array}{c}
    \det \Sigma_{13|}  = 1, \\
    \det \Sigma_{13|2} = \det \begin{psmallmatrix}
      1 & 0 \\ 2 & 2
    \end{psmallmatrix} = 2,
  \end{array} \quad
  \begin{array}{c}
    \det \Sigma_{23|}  = 2, \\
    \det \Sigma_{23|1} = \det \begin{psmallmatrix}
      2 & 0 \\ 1 & 2
    \end{psmallmatrix} = 4.
  \end{array}
\end{equation*}
Collecting all $(ij|K)$ which evaluate to zero gives the CI~structure
$\CIS{\Sigma} = \{ (12|) \}$.
\end{example}

This example shows that the CI~structure $\{ (12|) \}$ over ground set
$N = 123$ is realizable by a Gaussian distribution. Determining the
CI~structure of a given matrix is easy. The reverse ``synthesis problem''
of deciding whether a given CI~structure has a realization by a Gaussian
distribution involves intricate algebraic relations between the
principal and almost-principal minors of a generic positive-definite
matrix. An~algebraic proof that a CI~structure $\A$ is non-realizable
always comes with a valid inference rule $\varphi$ for Gaussians which
is not satisfied by~$\A$, such as in the next example.

\begin{example} \label{WeakTrans}
Let again $N = 123$ be a three-element ground set. To check whether the
CI~structure $\A = \{ (12|), (12|3) \}$ is realizable by a Gaussian
distribution, write a generic covariance matrix
\[
  \Sigma = \begin{pmatrix}
    p & a & b \\
    a & q & c \\
    b & c & r
  \end{pmatrix}
\]
and consider the equations imposed on $\Sigma$ by $\A$:
\begin{align*}
  0 = \det \Sigma_{12|} = a
  \quad\text{and}\quad
  0 = \det \Sigma_{12|3} = \det \begin{psmallmatrix}
    a & b \\
    c & r
  \end{psmallmatrix} = ar - bc.
\end{align*}
Together these equations imply $bc = 0$ and therefore $b = 0$ or $c = 0$.
This proves that the inference~rule
\[
  (12|) \wedge (12|3) \Rightarrow (13|) \vee (23|)
\]
is valid for all Gaussian distributions. Since $\A$ does not satisfy
this inference rule, it is not realizable by a Gaussian distribution.
It is easy to see that the two inclusion-minimal CI~structures which
extend $\A$ and which are realizable by a positive-definite matrix
are $\{ (12|), (12|3), (13|), (13|2) \}$ and $\{ (12|), (12|3),
(23|), (23|1) \}$ corresponding to the two alternative conclusions
of the inference~rule.
\end{example}

A \emph{gaussoid} is a subset of $\A_N$ which is closed under the
gaussoid axioms:
\begin{alignat}{4}
\label{eq:G1} \tag{$\G G$.i}   & (ij|L)  &&\wedge (ik|jL) &&\;\Leftrightarrow\; (ik|L)  &&\wedge (ij|kL), \\[-.2em]
\label{eq:G2} \tag{$\G G$.ii}  & (ij|kL) &&\wedge (ik|jL) &&\;\Leftrightarrow\; (ij|L)  &&\wedge (ik|L),  \\[-.2em]
\label{eq:G4} \tag{$\G G$.iii} & (ij|L)  &&\wedge (ij|kL) &&\;\Rightarrow\;     (ik|L)  &&\vee   (jk|L),
\end{alignat}
for all distinct $i, j, k \in N$ and $L \subseteq N \setminus ijk$.
The first two formulas decompose into $2 \cdot 2 \cdot 2 = 8$ inference
forms in total (some of which are redundant due to symmetries in the
axioms); the third formula is a single inference form. The
axiom~\eqref{eq:G1} defines semigraphoids, the ``$\Rightarrow$''
direction of~\eqref{eq:G2} is known as the intersection axiom.
These six inference rules together form the definition of a
\emph{graphoid}. The ``$\Leftarrow$'' direction of~\eqref{eq:G2},
the converse of intersection, is the composition axiom. The final
axiom~\eqref{eq:G4} is weak transitivity, a special case of which
was proved in \Cref{WeakTrans}. Thus gaussoids are the weakly
transitive, compositional graphoids, or the weakly transitive
semigaussoids, as coined in~\cite{DrtonXiao}.

The gaussoid axioms are valid inference rules for Gaussian distributions
on every ground set~$N$ by \cite[Corollary~1]{MatusGaussian}. They are
the ``simplest'' inference rules for Gaussians in the sense that they are
not only necessary but also sufficient for Gaussian realizability over
the three-element ground set, which is the smallest non-trivial size.
They are not sufficient for more than three variables.
The main result of this paper (\Cref{TwoAntecedental}) shows that every
valid inference rule which has at most two antecedents is implied by
the gaussoid axioms --- without restrictions on the finite ground set~$N$.

The symmetries of Gaussian CI~structures play a prominent role in the
proof of the main result in \Cref{sec:Proof}. The symmetric group $\Sym_N$
of permutations of~$N$ acts on the random vector $\xi = (\xi_i)_{i \in N}$
by relabeling the entries. This immediately translates to an action on
CI~statements~as
\[
  (ij|K)^\pi := (\pi(ij)|\pi(K)),
\]
for $\pi \in \Sym_N$. This action is extended element-wise to CI~structures.
The equivalence relation induced by this group action on CI~structures is
\emph{isomorphy}. On covariance matrices, the group acts accordingly by
simultaneous permutation of the rows and columns or, equivalently, a~permutation
of the axes of $\BB R^N$. This is an orthogonal coordinate change and thus
preserves positive-definiteness. Another important symmetry is the one
induced by \emph{duality} of CI~statements,
\vskip -0.9em
\[
  (ij|K)^* := (ij|N \setminus ijK),
\]
which swaps a conditioning set~$K$ with its complement in $N \setminus ij$.
On covariance matrices, duality corresponds to matrix inversion
$\Sigma \mapsto \Sigma^{-1}$ so that $\CIS{\Sigma^{-1}} = \CIS{\Sigma}^*$;
cf.~\cite[Lemma~1]{LnenickaMatus}.
Both, isomorphy and duality, are special cases of a larger symmetry group,
the hyperoctahedral group $\HypOct_N$, which is generated by the reflection
symmetries of the $N$-dimensional cube. The combinatorial motivation for
considering this group is explained in \cite{ConstructionMethods}, but it
is not important for the present work.
As an abstract group, $\HypOct_N$ equals the semidirect product
$(\BB Z/2)^N \rtimes \Sym_N$, i.e., each of its elements can be uniquely
written as a composition of a \emph{swap} from $(\BB Z/2)^N$ and a
\emph{permutation} from $\Sym_N$. Each vector in the group of swaps
$(\BB Z/2)^N$ is an indicator vector of a subset $Z \subseteq N$ and
acts on a CI~statement via
\vskip -0.8em
\[
  (ij|K)^Z := (ij|K \oplus (Z\setminus ij)),
\]
where $\oplus$ denotes the symmetric difference $A \oplus B = A \setminus B
\cup B \setminus A$. Duality is a special case of this action by swapping
everything, i.e.,~$Z = N$. The second constituent of $\HypOct_N = (\BB Z/2)^N
\rtimes \Sym_N$ are the permutations~$\Sym_N$, which simply act by isomorphy.

\begin{example}
Consider the CI~structure $\A = \{ (12|), (12|345), (34|), (34|25) \}$.
This structure satisfies the gaussoid axioms. Let $(1 \; 3)$ be the
cyclic permutation on $N = 12345$ which exchanges $1$ with $3$ and leaves
every other element fixed. The images of $(1 \; 3)$, duality and swap by
$Z = 123$ on $\A$ are, respectively:
\begin{align*}
  \A^{(1 \, 3)}     &= \{ (23|), (23|145), (14|), (14|25) \},   \\[-.2em]
  \A^* = \A^{12345} &= \{ (12|345), (12|), (34|125), (34|1) \}, \\[-.2em]
  \A^{123}          &= \{ (12|3), (12|45), (34|12), (34|15) \}.
\end{align*}
It is easy to see that the gaussoid axioms \eqref{eq:G1}--\eqref{eq:G4}
are invariant under the hyperoctahedral group. Thus, every CI~structure
obtained above from $\A$ by one of the group actions must be a gaussoid
as~well.
\end{example}

In contrast to gaussoids, realizable Gaussian CI~structures are invariant
only under isomorphy and duality, but not under the hyperoctahedral group.
The $\HypOct_N$ action can be defined on symmetric matrices but it does not
preserve positive-definiteness. This is explained in detail in
\Cref{sec:Group}. That section also extends the definition of realizability
via \eqref{eq:CIdef} to matrices over ordered fields other than~$\BB R$ and,
if the field is not ordered, to matrices whose principal minors are non-zero
instead of positive, giving rise to, respectively, \emph{positively} and
\emph{algebraically realizable} gaussoids over specific~fields. Algebraically
realizable CI~structures turn out to be invariant under the hyperoctahedral
group, which is a key ingredient to the proof of our main result.

\section{Statement of the main result}
\label{sec:MainResult}

On a fixed ground set $N$, each subset of $\A_N$ corresponds to an
assignment of truth values to Boolean variables indexed by~$\A_N$.
Every set of CI~structures (a set of truth assignments to the
elements of~$\A_N$) is the set of satisfying assignments to a Boolean
formula in conjunctive normal form (CNF) whose variables are in~$\A_N$.
We restrict attention to the clauses of these CNFs. Each clause $\varphi$
is a disjunction of negated and non-negated statements from $\A_N$ and
can be written in inference form
\[
  \varphi : \bigwedge_s (i_s j_s | K_s) \;\Rightarrow\; \bigvee_t (x_t y_t | Z_t).
\]
This the general form of CI~inference axioms, which are the subject
of this article.

\begin{definition}
Let $\BB A \supseteq \BB A\!^*$ be sets of CI structures over a fixed
ground set~$N$. $\BB A$~is a \emph{$k$-antecedental approximation} of
$\BB A\!^*$ if every inference form $\varphi$ with at most $k$ antecedents
and variables in $\A_N$ which is valid for $\BB A\!^*$ is also valid for
$\BB A$.
\end{definition}

We imagine $\BB A$ to be a simpler set approximating $\BB A\!^*$ from above.
Because of the inclusion $\BB A\!^* \subseteq \BB A$, every inference rule
which is valid for $\BB A$ also holds for $\BB A\!^*$. The definition above
concerns a degree~$k$ to which the converse holds. In this article, the
role of $\BB A$ is played by gaussoids and that of $\BB A\!^*$ by realizable
gaussoids, for different notions of realizability.
From an axiomatic point of view, one may also say that the axioms for~$\BB A$
are \emph{$k$-antecedentally complete} for the chosen notion of
realizabillity~$\BB A\!^*$.

Our proof of the two-antecedental approximation property of gaussoids relies
on a general principle which was also used in Studený's proof for discrete~CI.
A~\emph{minimal $\BB A$-extension} of a CI structure $\G A$ is a
CI~structure~$\G A'$ which is inclusion-minimal with the properties that
$\G A' \supseteq \G A$ and $\G A' \in \BB A$. In the world of discrete~CI
and semigraphoids, this minimal extension is unique, because both,
discretely realizable CI~structures and semigraphoids, are closed
under~intersection; but for Gaussians and gaussoids this is not~true,
as seen in \Cref{WeakTrans}.

\begin{lemma} \label{ApproxLemma}
Let $\BB A \supseteq \BB A\!^*$ be sets of CI structures over~$N$.
Then $\BB A$ is a $k$-antecedental approximation of~$\BB A\!^*$ if
every minimal $\BB A$-extension of every subset of $\A_N$ of cardinality
at most $k$ belongs to $\BB A\!^*$.
\end{lemma}

\begin{proof}
Let $\varphi: \bigwedge \G L \Rightarrow \bigvee \G M$ be a valid
inference rule for $\BB A\!^*$ with $|\G L| \le k$. We have to show
that $\varphi$ is valid for $\BB A$. Equivalently, letting $\BB A(\varphi)$
denote the subset of $\BB A$ which satisfies $\varphi$, we show that
$\BB A \subseteq \BB A(\varphi)$.

Consider any $\G A \in \BB A$. If the antecedents $\G L$ of $\varphi$
are not contained in $\G A$, then $\G A$ is vacuously contained in
$\BB A(\varphi)$. On the other hand, if $\G A$ contains $\G L$,
then it also contains a minimal $\BB A$-extension $\G L'$ of~$\G L$.
Since $|\G L| \le k$, the structure $\G L'$ belongs to $\BB A\!^*$
by assumption. Hence $\G L'$ satisfies $\varphi$, which means that
$\G L' \cap \G M \not= \emptyset$. Then $\G A$, containing $\G L'$,
also satisfies~$\varphi$.
\end{proof}

We can now state the main result:

\begin{mainthm} \label{MainThm}
Over every ground set, every minimal gaussoid extension of at most two
CI~statements is realizable by a positive-definite matrix with rational
entries, which can be picked arbitrarily close to the identity matrix.
\end{mainthm}

There is a notion of algebraic realizability over every field and of
positive realizability over every ordered field. In each case, testing
realizability of a fixed gaussoid means deciding if a system of polynomial
equations, inequations and (for positivity) inequalities with integer
coefficients has a solution. Since the rational numbers are the prime
field of characteristic zero, finding a positive and rational solution
is the hardest problem among all realizabilities over fields of
characteristic~zero.

The fact that the approximation is valid uniformly over all ground sets
is significant. Generally, the realizabilities considered here do not have
a finite axiomatization. This means that, as the ground set grows, evermore
logically independent inference rules become necessary to cut out the
realizable gaussoids. \Cref{MainThm} implies that among these new inference
rules for larger and larger ground sets, the easiest ones, up to two
antecedents, are all logical consequences of the well-known gaussoid axioms.
This is the titular

\begin{corollary} \label{TwoAntecedental}
Gaussoids are two-antecedental approximations of algebraic and of
positive Gaussian conditional independence structures over characteristic
zero.
\qed
\end{corollary}

The realizability proofs for~\Cref{MainThm} are composed of two ideas
to be presented in the next two sections, respectively.
The first idea is to relax the positive-definiteness requirement
and study instead the aforementioned algebraic Gaussians over general fields.
The resulting CI structures are still gaussoids and they are closed under
the action of the hyperoctahedral group which allows passing to an easier
orbit representative in realizability proofs.
The~second idea, compensating the previous relaxation, is to study more
special realizations, namely rational parametrizations of spaces of
matrices.
Formally, such a space is represented by a matrix whose entries are
elements of the rational function field $\BB Q(\eps_1, \dots, \eps_p)$
with infinitesimal variables $\eps_k$. In this space, the algebraically
realized gaussoid over $\BB Q$ is determined uniquely by the rational
functions for all sufficiently small values of the~$\eps_k$.
If the parametrized matrices converge to a positive-definite matrix as
the~$\eps_k$ tend to zero, then matrices in the interior of this space
are positive-definite for small enough~$\eps_k$. In this way, positive
realizability of a gaussoid is recovered from an algebraic construction
and inspection of a limit.
In the proof of the main theorem in~\Cref{sec:Proof}, both techniques
are applied to minimal gaussoid extensions of at most two CI statements.
We turn the problem of finding positive-definite realizations around and
instead find spaces of matrices realizing \emph{one} easy hyperoctahedral
representative of each gaussoid orbit converging to \emph{every}
hyperoctahedral image of the identity matrix. Then, given any gaussoid
in the representative's orbit, we apply the inverse group action to the
right space of matrices so that the transformed space approaches the
positive-definite identity matrix and realizes the given gaussoid.
This yields the desired rational regular Gaussian realizations.

\section{Algebraic Gaussians and the hyperoctahedral group}
\label{sec:Group}

The gaussoid axioms were derived in~\cite[Corollary~1]{MatusGaussian}
as consequences of an identity among minors of a complex symmetric
matrix, where the proof of each instance of an axiom requires certain
principal minors of the matrix not to vanish. It follows that the
gaussoid axioms hold for the set $\CIS{\Gamma} = \{ (ij|K) \in \A_N :
\det \Gamma_{ij|K} = 0 \}$ of vanishing almost-principal minors of
any symmetric matrix~$\Gamma$, provided that all of its principal
minors are non-zero. Such matrices are \emph{principally regular}.
In this article positive-definite and principally regular matrices
are implicitly symmetric.
Moreover, $\Gamma$ can have entries from any field~$\BB K$ because,
while Matúš's result is stated for complex matrices only, special
structure of the complex numbers is only invoked in a continuity
argument which is circumvented by the assumption of principal
regularity.

Principal regularity is not only an obvious substitute for
positive-definiteness over arbitrary fields. It is the technical
condition which allows the formation of all Schur complements of
the matrix, which correspond to conditional distributions in the
positive-definite setting. The~property is inherited by the inverse
matrix, by principal submatrices and Schur complements and therefore
is precisely what is required to salvage the theory of minors of
regular Gaussians; see~\cite{ConstructionMethods}.

\begin{definition} \label{Realizable}
A gaussoid $\G G$ is \emph{algebraically realizable (over $\BB K$)}
if there is a principally regular matrix $\Gamma$ over $\BB K$ such
that $\G G = \CIS{\Gamma}$. If in addition $\BB K$ is an ordered
field and $\Gamma$ has only positive principal minors, then $\G G$
is \emph{positively realizable (over $\BB K$)}.
By slight abuse of language we refer to algebraically realizable
gaussoids as well as their realizing matrices as \emph{algebraic
Gaussians} and similarly for \emph{positive Gaussians}.
\end{definition}

In this terminlogy, the familiar multivariate Gaussian distributions
are positive Gaussians over the real-closed field~$\BB R$. Every
algebraic Gaussian is, by Matúš's corollary, a gaussoid. Conversely,
given any gaussoid $\G G$, its \emph{algebraic realization space} over
a field $\BB K$ is the set of matrices with entries in $\BB K$ that
algebraically realize~$\G G$. It~is specified by the vanishing of the
almost-principal minors in $\G G$, the non-vanishing of the
almost-principal minors not in $\G G$ and the non-vanishing of all
principal minors of a symmetric matrix. The space of these matrices
is a constructible set. The~\emph{positive realization space} over
an ordered field refines the non-vanishing of principal minors
into positivity constraints, resulting in~a~semialgebraic~set.

\begin{remark} \label{ImplicationAlgebraic}
With a fixed notion of realizability, the \emph{CI implication problem}
asks to decide if a given inference form is valid for the class of
realizable CI~structures. There is much previous work about CI~implication
for graphical models and approximations or special cases of discrete~CI;
see for instance~\cite{DAGCausal,GeigerPearlLogical,EfficientCI,LatticeCI}
as well as \cite{InformationDecision} from the point of view of information
theory, and the references therein.
For Gaussians over algebraically closed fields, there is an algebraic
characterization of CI~implication which was observed by~Matúš.
His~result~\cite[Proposition~1]{MatusGaussian} is stated for the
complex numbers but again the proof works over arbitrary algebraically
closed fields, with the crucial ingredient being Hilbert's Nullstellensatz.
The~geometric summary of this characterization is that
$\bigwedge \G L \Rightarrow \bigvee \G M$ holds if and only if in the
affine space of symmetric matrices $\Gamma$, the polynomial
$f_{\G M} = \prod_{(ij|K) \in \G M} \det \Gamma_{ij|K}$ vanishes on the
constructible set that is defined by the vanishing of the almost-principal
minors indexed by $\G L$ and the non-vanishing of the principal minors
of~$\Gamma$, because then on every point on this constructible set
defined by $\G L$ at least one of the minors in $\G M$ vanishes, proving
the validity of the inference.
This set of matrices is obtained as a projection of a variety
in a higher-dimensional affine space by a standard construction and hence
the CI~implication problem amounts to a radical~ideal membership test.
The same idea paired with the Real Nullstellensatz yields a characterization
for algebraic realizability over real-closed fields, employing the real
variety and hence the real radical of the analogously defined ideal.
Finally, the Positivstellensatz can be used similarly to characterize the
CI~implication problem for positive Gaussians over a real-closed field in
terms of an ideal membership query on the radical ideal of the preorder
associated to the positive realization space of~$\G L$.
The reader is referred to~\cite{CoxLittleOShea} and~\cite{Marshall} for
background information on geometry over algebraically and, respectively,
real-closed fields. These characterizations of CI~implication directly
yield a procedure for realizability testing. Namely, a CI structure~$\G L$
is realizable if and only if it does not appear as the antecedent set
of a non-trivial valid inference rule, the weakest of which, and hence
the only one that needs to be tested, is~$\bigwedge \G L \Rightarrow
\bigvee (\A_N \setminus \G L)$.
\end{remark}

Algebraic realizations have two advantages over positive realizations.
The first advantage is that algebraic realizability over algebraically
closed fields tends to be easier to decide in practice, as outlined
in the previous remark, using Hilbert's Nullstellensatz and Gröbner bases,
than positive realizability over a real-closed field, requiring a
Positivstellensatz certificate for non-realizability. Over an ordered
field, a positively realizable gaussoid is always algebraically
realizable.

The second advantage, paramount to the proof of our main theorem,
is that algebraic realizability is preserved under the hyperoctahedral
group introduced in~\Cref{sec:Prelim}.
The hyperoctahedral group action on realizing matrices is a quotient
of the group $(\BB Z/4)^N \rtimes \Sym_N$. This group is, in turn,
a discrete subgroup of the $\SL_2(\BB R)^N$ action on the
Lagrangian Grassmannian; cf.\ \cite{HoltzSturmfels,Geometry}.
In the semidirect product, every group element can be written as the
composition of an element of $\Sym_N$ and one of $(\BB Z/4)^N$.
The permutation part is just an orthogonal coordinate change, permuting
rows and columns of the matrix, corresponding to the $\Sym_N$ action
on CI structures (almost-principal minors) and merely permuting
the set of principal minors. This action changes neither principal
regularity nor positive-definiteness and therefore we focus on the
$(\BB Z/4)^N$ part in the remainder of this section.

Our account of the $(\BB Z/4)^N$ symmetry is based on~\cite[Section~3]{Geometry}
but we give more details about the action on matrices. Each $\BB Z/4$
factor in the $N$-fold product making up the group is represented by
the four $2 \times 2$ matrices
\[
  \BB Z/4 =
  \left\{
    \begin{pmatrix}
      1 & 0 \\ 0 & 1
    \end{pmatrix},
    \begin{pmatrix}
      0 & 1 \\ -1 & 0
    \end{pmatrix},
    \begin{pmatrix}
      -1 & 0 \\ 0 & -1
    \end{pmatrix},
    \begin{pmatrix}
      0 & -1 \\ 1 & 0
    \end{pmatrix}
  \right\}.
\]
To each tuple $(X_1, \dots, X_n) \in (\BB Z/4)^N$ we associate four
$N \times N$ diagonal matrices $A, B, C, D$ such that
\[
  X_i = \begin{pmatrix}
    A_{ii} & B_{ii} \\
    C_{ii} & D_{ii}
  \end{pmatrix}.
\]
The image of a symmetric matrix $\Gamma$ under $(X_1, \dots, X_n)$ is
$\Gamma' = (A + \Gamma C)^{-1}(B + \Gamma D)$. $\Gamma'$ is again symmetric
by~\cite[Lemma~13]{HoltzSturmfels} and the following \Cref{HypOctMinors}
describes its principal and almost-principal minors.
To facilitate this description we use a parametrization of this group
action. For any subset $Z \subseteq N$ and a tuple of signs $\delta \in
\{\pm 1\}^N$, choose the group element $(X_1, \dots, X_n)$ where
\[
  X_i = \delta_i \begin{cases}
    \begin{psmallmatrix} 1 &  0 \\ 0 & 1 \end{psmallmatrix}, & \text{$i \not\in Z$}, \\
    \begin{psmallmatrix} 0 & -1 \\ 1 & 0 \end{psmallmatrix}, & \text{$i \in Z$}. \\
  \end{cases}
\]
Then the action can be written as $\G S_Z^\delta(\Gamma) :=
(A + \Gamma C)^{-1}(B + \Gamma D)$ with
\begin{equation*}
\begin{aligned}[c]
  A_{ii} =  D_{ii} &= \begin{cases}
    \delta_i, & \text{$i \not\in Z$}, \\
    0, & \text{$i \in Z$},
  \end{cases}
\end{aligned}
\qquad\qquad
\begin{aligned}[c]
  C_{ii} = -B_{ii} &= \begin{cases}
    0, & \text{$i \not\in Z$}, \\
    \delta_i, & \text{$i \in Z$}.
  \end{cases}
\end{aligned}
\end{equation*}
In expressing minors of $\G S_Z^\delta(\Gamma)$ in terms of $Z$, $\delta$
and $\Gamma$, it becomes necessary to recombine the involved subsets of $N$.
Using the abbreviations $AB = A \cup B$ and $[AB] = A \cap B$ as well as
$A^\comp$ for the complement of $A$ in $N$, any combination of interest
can be efficiently written down in ``disjunctive normal form''.
For~example, $[ZK^\comp][Z^\comp K] = (Z \cap K^\comp) \cup (K \cap
Z^\comp) = (Z \setminus K) \cup (K \setminus Z) = Z \oplus K$.

One subtlety in the following statement and its proof concerns the signs
of minors. In order to have a well-defined sign of the determinant, we
employ the following \emph{sign convention} about the ordering of rows
and columns in submatrices. Because the sign is invariant under simultaneous
permutations of the rows and columns, the absolute ordering is not important,
but instead which row is matched with which column in any ordering of the
row and column sets. Instead of imposing an absolute order on the set $N$,
it will be convenient to pair every $k \in K$ with itself, in principal
minors with respect to $K$, and in almost-principal minors $(ij|K)$ to
additionally pair $i$~and~$j$.

\begin{proposition} \label{HypOctMinors}
Let $\Gamma$ be principally regular over $\BB K$ and $Z \subseteq N$ and
$\delta \in \{\pm 1\}^N$ be arbitrary. Then $\Gamma' = \G S_Z^\delta(\Gamma)$
is principally regular over $\BB K$. The gaussoid $\CIS{\Gamma'} = \CIS{\Gamma}^Z
= \{ (ij|(Z \setminus ij) \oplus K) : (ij|K) \in \CIS{\Gamma} \}$.
More precisely, we have the following formulas for the principal and
almost-principal minors of $\Gamma'$:
\begin{align*}
  \Gamma'_K      &= (-1)^{[ZK]} \det \Gamma_Z^{-1} \cdot \det \Gamma_{Z \oplus K}, \\
  \Gamma'_{ij|K} &= (-1)^{[ZK]} \det \Gamma_Z^{-1} \cdot \delta_i \delta_j \det \Gamma_{ij|(Z \setminus ij) \oplus K}.
\end{align*}
\end{proposition}

\begin{proof}
The matrix $\Gamma'$ satisfies the matrix equation $(A + \Gamma C) \Gamma' =
B + \Gamma D$ and hence its minors can be computed with a generalized Cramer's
rule~\cite{GeneralizedCramer}:
\[
  \det \Gamma'_{I,J} = \det(A + \Gamma C)^{-1} \det \left[
    (A + \Gamma C)(I,J : B + \Gamma D)
  \right],
\]
for sets $I, J \subseteq N$ of the same size and where $X(I,J : Y)$ denotes
the matrix $X$ where the columns indexed by $I$ are replaced by the columns
of $Y$ indexed by~$J$. In this notation, we omit $J$ if it equals~$I$.
In~addition we use $\delta X$ to denote the matrix $X$ where the $i$th column
is scaled with $\delta_i$.

By definition of $A$ and $C$ we have $A + \Gamma C = \delta \Gamma(Z^\comp : I_N)$
and by Laplace expansion on the unit columns in $Z^\comp$,
\[
  \det(A + \Gamma C) = \prod_k \delta_k \cdot \det \Gamma_Z.
\]

To compute the principal minor for $K \subseteq N$, notice that
\[
  (A + \Gamma C)(K : B + \Gamma D) = \delta [\Gamma(Z^\comp : I_N)](K : \Gamma(Z : -I_N))
    =: \delta \Gamma''.
\]
The columns of $\Gamma''$ are composed as follows:
\begin{description}
\item[{$[ZK^\comp]$}] respective columns of $\Gamma$,
\item[{$[ZK]$}] respective negative unit vectors,
\item[{$[Z^\comp K^\comp]$}] respective unit vectors,
\item[{$[Z^\comp K]$}] respective columns of $\Gamma$.
\end{description}
By Laplace expansion of the unit vector columns, we obtain
\[
  \det \delta \Gamma'' =
    \prod_k \delta_k \cdot (-1)^{[ZK]} \det \Gamma_{[Z K^\comp][Z^\comp K]},
\]
and $[Z K^\comp][Z^\comp K] = Z \oplus K$ proves the principal minor
formula. Notice that the Laplace expansions deleted the same rows and
columns, so the sign convention is preserved.

For the almost-principal minor given by $(ij|K)$, the same procedure yields
the columns of $\Gamma''$
\begin{description}
\item[{$[Z(iK)^\comp]$}] respective columns of $\Gamma$,
\item[{$[ZK]$}] respective negative unit vectors,
\item[{$[Z^\comp (iK)^\comp]$}] respective unit vectors,
\item[{$[Z^\comp K]$}] respective columns of $\Gamma$,
\item[$i$] the $j$th column of $\Gamma(Z : -I_N)$.
\end{description}
Pulling out the $\delta$ signs from the determinant, we get the sign
$\prod_{k \not= i,j} \delta_k \cdot \delta_j^2 = \delta_i \delta_j
\prod_k \delta_k$ because the $i$th column's $\delta_i$ was replaced
by a $j$th column's~$\delta_j$. In addition, performing again Laplace
expansion on the unit vector columns, the determinant so far is
\begin{align*}
\label{eq:Gamma''minor}
\tag{$\diamond$}
  \det \delta \Gamma'' = \delta_i \delta_j \prod_k \delta_k \cdot
    (-1)^{[ZK]} \det \Gamma''_{i[Z(iK)^\comp][Z^\comp K]}.
\end{align*}
The contents of column $i$ of $\Gamma''$ depends on whether $j \in Z$ or not.
If $j \not\in Z$, then the $i$th column of $\Gamma''$ is just the $j$th
column of $\Gamma$. Thus the remaining minor of $\Gamma''$
in~\eqref{eq:Gamma''minor} is revealed to be the almost-principal minor of
$\Gamma$ with row indices $i[Z(iK)^\comp][Z^\comp K] = i[Z(ijK)^\comp][Z^\comp K]$
and column indices $j[Z(ijK)^\comp][Z^\comp K]$. It~is easy to see that the
replacement of column $i$ by column $j$ leaves the rows and columns
correctly paired and it remains to compute the conditioning set as
$[Z(ijK)^\comp][Z^\comp K] = (Z \setminus ij) \oplus K$.

If $j \in Z$, then the $i$th column contains the negative $j$th unit vector.
Laplace expansion with respect to this column results in the column labeled
$i$ and the row labeled $j$ to be removed and incurs a sign change which
depends on the distance between these columns. By simultaneously reordering
rows and columns, we can assume that rows and columns $i$~and~$j$ are next
to each other. In this case, the sign change is~$-1$, which is compensated
by the \emph{entry}~$-1$ in the eliminated column. The reordering ensures
that rows and columns are properly paired after Laplace expansion.
The remaining minor of $\Gamma''$ has row indices
$i[Z(iK)^\comp][Z^\comp K] \setminus j = i[Z(ijK)^\comp][Z^\comp K]$
and column indices $[Z(iK)^\comp][Z^\comp K] = j[Z(ijK)^\comp][Z^\comp K]$
and is thus again the almost-principal minor $(ij|(Z \setminus ij) \oplus K)$
of~$\Gamma$.
\end{proof}

These formulas describe in particular all the entries of $\G S_Z^\delta(\Gamma)$
in terms of $\Gamma$, $Z$ and~$\delta$. Remarkably, the choice of $\delta$
has no influence at all on the principal minors, and only changes the sign
of almost-principal ones.
Hence, by identifying in each $\BB Z/4$ factor the two matrices with opposite
signs, we obtain a quotient group isomorphic to $(\BB Z/2)^N \rtimes \Sym_N$
which faithfully implements the hyperoctahedral group on algebraic gaussoids
over any field. The realizing matrix may not be well-defined but the quotient
is conclusive about its positivity, over ordered fields, and thus can be used
to certify positive realizability of hyperoctahedral images of gaussoids.

\begin{remark} \label{RealizDiagonals}
If $\Gamma$ is a principally regular matrix over a field $\BB K$ and
the diagonal entries $\Gamma_{ii}$ are squares in $\BB K$, then the
diagonal matrix $D$ with entries $1 / \sqrt{\Gamma_{ii}}$ can be formed
over $\BB K$.
It is easy to check that $D \Gamma D$ is a principally regular matrix
which realizes the same gaussoid as $\Gamma$ and has a 1-diagonal.
This argument allows to remove degrees of freedom from the generic
candidate matrix in realizability~tests.

For example, since the complex numbers are closed under square roots,
every algebraic Gaussian over~$\BB C$ has a realization with unit diagonal.
The real numbers are missing $\sqrt{-1}$, so algebraic realizations of
gaussoids over~$\BB R$ can only be assumed to have $\pm 1$ entries on
the diagonal. The restriction to positive realizations restores the
expectation of a $1$-diagonal. Finally, the rational numbers are missing
many square roots, but over~$\BB Q(\sqrt{-1})$ one can at least assume
positive squarefree integers on the diagonal of an algebraic realization.
\end{remark}

\section{Infinitesimally perturbed realizations}
\label{sec:Epsilon}

The theme of this section is perturbing principally regular matrices
to find related but more generic CI~structures in their neighborhood.
By ``more generic'' we mean that fewer almost-principal minors vanish,
which implies smaller gaussoids. Our target, the minimal gaussoid
extensions of at most two CI statements, are evidently among the smallest
possible gaussoids. The focus lies on perturbing the \emph{least} generic
matrices from this point of view (modulo the algebraic torus action
described in \Cref{RealizDiagonals}), which lie in the hyperoctahedral
orbit of the identity matrix.
Formally, our main tool consists of passing to Gaussians over the (ordered)
field of rational functions $\BB K(\eps_1, \dots, \eps_p)$, extending
$\BB K$ with finitely many (infinitesimal) variables.
This~technique already appears in~\cite[Theorem~1]{LnenickaMatus}
as well as~\cite[Theorem~3.3]{SimecekNonfinite} and is on vivid display
in~\cite[Table~1]{LnenickaMatus} --- but always without the formalization
of the concept of positive or algebraic Gaussians over function fields.

By Tarski's Transfer principle~\cite[Theorem~1.4.2]{Marshall} and the
fact that Gaussian realization spaces are described by polynomial
constraints with integer coefficients, every gaussoid which is algebraically
or positively realizable over an ordered field extension of $\BB R$ is
also algebraically or, respectively, positively realizable over~$\BB R$.
This principle applies to the constructions in this section, but the field
of rational functions is special: if a realization over $\BB Q(\eps_1, \dots, \eps_p)$
is found, then plugging in sufficiently small rational numbers for the
$\eps_k$ yields even a rational realization which is not promised by
the Transfer principle. Therefore this section circumvents the Transfer
principle via \Cref{RationalFunctions} and emphasizes rational constructions.

\begin{lemma} \label{RationalFunctions}
Consider the following two situations:
\begin{enumerate}[label=(\arabic*)]
\item
$\BB K$ is an infinite field and $\BB L = \BB K(x_1, \dots, x_p)$
the field of rational functions in variables $x_1, \dots, x_p$
over~$\BB K$.

\item
$\BB K$ is an ordered field and $\BB L = \BB K(\eps_1, \dots, \eps_p)$
the ordered field of rational functions in infinitesimals $0 < \eps_1 <
\dots < \eps_p$ over~$\BB K$.
\end{enumerate}
In both situations, a gaussoid is realizable over $\BB L$ if and only if
it is already realizable over~$\BB K$.
\end{lemma}

\begin{proof}
One inclusion is obvious by the inclusion of fields.
In the other direction, it~suffices to show how to adjoin one variable~$x$
or one infinitesimal~$\eps$, so the proof proceeds by induction on~$p$.

\begin{paraenum}
\item
Let $\BB K$ be an infinite field and $\Gamma$ principally regular over~$\BB L
= \BB K(x)$.
The~CI~structure $\CIS{\Gamma}$ is defined by vanishing and non-vanishing
constraints on principal and almost-principal minors of $\Gamma$. These are
polynomials in the entries of $\Gamma$ and therefore rational functions
over~$\BB K$. If a rational function $f \in \BB L$ is zero in $\BB L$,
then every evaluation $f(a)$ for $a \in \BB K$ is zero. Otherwise $f$ has
non-zero numerator and denominator, which are univariate polynomials over
$\BB K$. The zeros of the denominator are poles of the evaluation of~$f$
and the zeros of the numerator are the zeros of~$f$. Both zero sets are
finite. Since $\BB K$ is infinite, one can find a point $a \in \BB K$
avoiding all the undesirable poles and zeros of all principal minors and
of all non-zero almost-principal minors of $\Gamma$ simultaneously.
On this point~$a$, the evaluation $\Gamma(a)$ is a principally regular
matrix over $\BB K$ and $\CIS{\Gamma} = \CIS{\Gamma(a)}$.

\item
Suppose $\BB K$ is ordered. This implies that its characteristic is zero
and in particular that it is infinite. Let $\Gamma$ positively realize a
CI~structure over~$\BB L = \BB K(\eps)$. Again we seek a positive realization
of $\CIS{\Gamma}$ over~$\BB K$ by plugging in elements of $\BB K$ for~$\eps$.
By~the previous part of the proof, the ``algebraic part'' of positive
realizability, i.e., the vanishing and non-vanishing conditions of
almost-principal minors, is satisfied on all but finitely many points.
It remains to find infinitely many points of $\BB K$ on which all principal
minors of $\Gamma$ evaluate to positive elements of $\BB K$. By~the
hypothesis, the principal minors of $\Gamma$ are positive in the ordering
of~$\BB L$. We can assume that numerators and denominators are both positive.
By~the ordering of~$\BB L$, a~polynomial $f \in \BB K[\eps]$ is positive
if and only if its lowest-degree non-zero coefficient is positive in~$\BB K$.
It~is then an easy exercise in ordered fields to construct $a_f > 0$
in~$\BB K$, depending on the coefficients and degree of~$f$, such that every
evaluation of $f$ on the open interval $(0, a_f)$ is positive in~$\BB K$.
\linebreak
Let $a^*$ be the minimum of the finitely many $a_f$ constructed for all
the (numerators and denominators of) principal minors of~$\Gamma$.
Since $a^* > 0$, the interval $(0, a^*)$ is infinite, and all but finitely
many evaluations of $\Gamma$ on this interval yield a positive realization
of $\CIS{\Gamma}$ over~$\BB K$.
\qedhere

\end{paraenum}
\end{proof}

\begin{remark}
The proof of the first part works, moreover, for all finite fields of
sufficient size. A lower bound can be given based on the number of roots
that the univariate polynomials in question have, which can be bounded
by the size of a concrete realizing matrix and the maximal degree of
its entries.
\end{remark}

For geometric intuition, suppose that $\BB K = \BB R$ for the moment.
Inspection of the previous proof then paints the following picture of our
main construction technique: we~define a space of real matrices parametrized
by rational functions in variables $\eps_1, \dots, \eps_p$. In fact,
we can replace all infinitesimals by sufficiently large and far apart
powers of a single infinitesimal and imagine just a \emph{curve segment}
of matrices parametrized by~$\eps$.
By~the defining rational functions, we control the algebraically realized
CI~structure on this curve, and as $\eps$ tends to zero, the matrices may
approach a limit matrix whose principal regularity or positive-definiteness
carries over to them by continuity. In~this way a certificate for algebraic
or positive realizability of the CI structure on the curve over the base
field~$\BB K$ is obtained.

The appeal to continuity and limits can be avoided by an easy sufficient
condition which is the subject of the next definition and which holds for
all our later applications. Let $\BB L$ be a rational function field over
$\BB K$ as in \Cref{RationalFunctions} and $\Gamma$ a matrix over~$\BB L$.
Suppose that the denominators of all of its entries have a non-zero constant
term. Then the evaluation $\Gamma^\circ := \Gamma(0, \dots, 0)$ is a matrix
over~$\BB K$. Notice that each minor of $\Gamma^\circ$ is the constant term
of the corresponding minor of $\Gamma$ and thus principal regularity and
positive-definiteness of $\Gamma^\circ$ over $\BB K$ imply that of $\Gamma$
over $\BB L$, allowing application of \Cref{RationalFunctions}.

\begin{definition}
Let $\BB K$ be a field and $\Gamma_0$ a principally regular matrix.
A gaussoid $\G G$ is \emph{realizable near $\Gamma_0$} if there exists
$\Gamma$ over $\BB K(\eps_1, \dots, \eps_p)$ such that $\G G = \CIS{\Gamma}$
and $\Gamma^\circ = \Gamma_0$.
If $\Gamma$ can be chosen over the prime field $\BB K = \BB Q$, we add
the adverb \emph{rationally} for additional emphasis.
\end{definition}

Among all (ordered) fields of characteristic zero, \emph{rational} (positive)
realizability is the most difficult to achieve, so whenever this is possible
it is highlighted in the results below. Realizability near a positive-definite
matrix immediately implies positive realizability by \Cref{RationalFunctions}.
We now assume that $\BB K$ is infinite or ordered and apply this
\namecref{RationalFunctions} to prove general construction methods for
Gaussians which reduce the case work in~\Cref{sec:Proof}.

\begin{lemma} \label{EpsilonSum}
Let $\G G = \CIS{\Sigma}$ and $\G H = \CIS{\Gamma}$ be algebraic Gaussians
over infinite~$\BB K$ on disjoint ground sets $N$ and~$M$, respectively.
Then $\G G \cup \G H$ is an algebraic Gaussian over~$\BB K$ on the ground
set $N \cup M$. It is realizable near the block-diagonal matrix containing
$\Sigma$ and $\Gamma$ blocks.
\end{lemma}

\begin{proof}
Define an $(NM \times NM)$-matrix
\begin{equation*}
  \Phi = \begin{pmatrix}
    \Sigma & \bm \eps \\
    \bm \eps^T & \Gamma
  \end{pmatrix},
\end{equation*}
where $\Sigma$ sits in the $N \times N$ block, $\Gamma$ in the $M \times M$
block and $\bm\eps = (\eps_{ij})_{ij \in N \times M}$ consists of independent
variables. Obviously $\Phi^\circ$ is the block-diagonal matrix from the
statement of the \namecref{EpsilonSum} and $\Phi$ has coefficients in~$\BB K$.
The~matrix $\Phi$ defines a realizable gaussoid $\CIS{\Phi}$ which restricts
to~$\G G$ on~$N$ and to~$\G H$ on~$M$. It remains to see that $\CIS{\Phi}$ contains
no other CI statement $(ij|K)$ where we decompose $K = N'M'$ with $N' \subseteq N$,
$M' \subseteq M$. We apply \Cref{RationalFunctions} and show that $\det \Phi_{ij|K}$
vanishes in $\BB K(\eps_{ij})$ only if $\G G$ or $\G H$ mandate it.

First assume $ij \subseteq N$ and $M' \not= \emptyset$ (as $M' = \emptyset$
yields that $\det \Phi_{ij|K} = \det \Sigma_{ij|N'}$, whose vanishing
depends only on $\G G$). The almost-principal submatrix $\Phi_{ij|K}$ is
written with rows labeled in order by $iN'M'$ and columns $jN'M'$:
\begin{align*}
  \det \Phi_{ij|K} &= \det \begin{pmatrix}
    \mbox{\normalfont\LARGE $\Sigma_{ij|N'}$} & \begin{matrix}
      \bm\eps_{i,M'} \\ \bm\eps_{N',M'}
    \end{matrix} \\
    \begin{matrix}
      \bm\eps_{M',j} & \bm\eps_{M',N'}
    \end{matrix} & \Gamma_{M'}
  \end{pmatrix} \\ &=
  \det\left(\Gamma_{M'}\right)
  \det\left(
    \Sigma_{ij|N'} - \begin{pmatrix} \bm\eps_{i,M'} \\ \bm\eps_{N',M'} \end{pmatrix}
      \Gamma_{M'}^{-1}
      \begin{pmatrix} \bm\eps_{M',j} & \bm\eps_{M',N'} \end{pmatrix}
  \right).
\end{align*}
The first determinant is a principal minor of $\Gamma$ and hence non-zero.
It suffices to show that the determinant of the Schur complement expression
is not the zero polynomial. For row $a \in iN'$ and column $b \in jN'$
the corresponding entry of the Schur complement is
\[
  \Sigma_{ab} - \sum_{k,l \in M'} (\Gamma_{M'}^{-1})_{kl} \eps_{al}\eps_{bk}
\]
and hence by the Leibniz formula the determinant equals
\[
  \sum_{\substack{\sigma: iN' \to jN' \\ \text{bijective}}}
  \sgn(\sigma) \prod_{a \in iN'} \left(
    \Sigma_{a\sigma(a)} - \sum_{k,l \in M'} (\Gamma_{M'}^{-1})_{kl} \eps_{al}\eps_{\sigma(a)k}
  \right).
\]
By assumption there exists $m \in M'$. In this multivariate polynomial
the coefficient of $\eps_{im}\eps_{jm}$ is
\begin{align*}
  \pm \sum_{\sigma \in \Sym_{N'}} \sgn(\sigma) (\Gamma_{M'}^{-1})_{mm}
    \prod_{a \in N'} \Sigma_{a\sigma(a)}
  &= \pm (\Gamma_{M'}^{-1})_{mm} \det \Sigma_{N'} \not= 0,
\end{align*}
which shows that $(ij|K) \not\in \CIS{\Phi}$ by \Cref{RationalFunctions} in
case $ij \subseteq N$ and $M' \not= \emptyset$. The~same proof applies to the
symmetric case with $N$ and $M$ exchanged.

The remaining case has $i \in N$ and $j \in M$. Then, by Laplace expansion
\begin{align*}
  \det \Phi_{ij|K} = \det \begin{pmatrix}
    \eps_{ij} & \cdots \\
    \vdots & \Phi_K
  \end{pmatrix} =
  \eps_{ij} \det \Phi_K \mp \dots,
\end{align*}
where $\det \Phi_K$ is non-zero and all other terms do not involve the
variable $\eps_{ij}$.
\end{proof}

\begin{lemma} \label{Embedding}
Let $\BB K$ be infinite and $\G G = \CIS{\Gamma}$ an algebraic Gaussian
over~$\BB K$ on~$N$. If~$L \cap N = \emptyset$, then ${\{ (ij|K) \in \A_{NL} :
(ij|K) \in \G G \}}$ and ${\{ (ij|KL) \in \A_{NL} : (ij|K) \in \G G \}}$ are
both algebraic Gaussians over~$\BB K$ on $NL$. They are realizable near
every block-diagonal matrix whose $N$-block is $\Gamma$ and whose $L$-block
is principally regular.
\end{lemma}

\begin{proof}
The first assertion follows directly from \Cref{EpsilonSum} using that
the empty gaussoid on $L$ is realizable near every principally regular
matrix, by introducing independent variables for each entry. To show
the second assertion, we make use of duality: if $\G G = \CIS{\Gamma}$
for $\Gamma$ principally regular, then $\Gamma^{-1}$ is principally
regular and realizes the dual gaussoid $\CIS{\Gamma^{-1}} = \G G^* =
\{ (ij|N \setminus ijK) : (ij|K) \in \G G \}$; cf.~\cite{LnenickaMatus,Geometry}.
Thus we may take $\G G^*$ over $N$ and embed it into $NL$ by the first
part of the proof, denoting the result by $\ol{\G G^*}$, and take its
dual over $NL$. All these operations preserve algebraic realizability
near a chosen block-diagonal matrix and we get
\[
  (ij|K) \in \G G \; \Leftrightarrow \; (ij|N \setminus ijK) \in \G G^*
  \; \Leftrightarrow \; (ij|NL \setminus ijKL) \in \ol{\G G^*}
  \; \Leftrightarrow \; (ij|KL) \in \ol{\G G^*}^*.
\]
No other CI statements over $NL$ arise because duality and embedding
preserve also cardinality.
\end{proof}

\begin{remark}
The ``dependent sum'' construction from \Cref{EpsilonSum} can be seen as
a perturbation of the \emph{direct sum} of semigraphoids~\cite{MatusMatroids,MatusClassification}
which is recovered in the limit $\eps_{ij} = 0$. Both constructions preserve
the properties of being a gaussoid and being a realizable gaussoid for
all the various notions of ``realizable'' covered in the above statements.
The dependent sum yields the most generic gaussoid on $NM$ which restricts
to its summands on $N$ and $M$, in that it satisfies no additional
relations. For~realizable gaussoids, this corresponds to the ``most
dependent'' joint Gaussian.
Similarly, \Cref{Embedding} shows that marginalization and conditioning
of realizable gaussoids from $NL$ to $N$ can be inverted generically,
inducing no further independencies over~$NL$.
\end{remark}

It easily follows from \Cref{HypOctMinors} that the quotient action
$\G S_Z$ on the $(\BB Z/4)^N$-orbit of the identity matrix, where
components $X_i$ with different signs $\delta_i$ are identified,
produces well-defined matrices, independent of the choice $\delta_i$
of representatives. This orbit consists of all $2^n$ diagonal matrices
with entries $(\pm 1, \dots, \pm 1)$. For any matrix~$J$ in this orbit
the action $\G S_Z(J)$ flips the signs of the diagonal entries of~$J$
indicated by~$Z$. The action of $\Sym_N$ does not leave this set of
matrices either, so it constitutes an orbit under the hyperoctahedral
group $\HypOct_N$.

Realizability near the identity matrix or its hyperoctahedral images
is a well-behaved notion in the theory of CI~structures: take a principally
regular matrix $\Gamma$ over $\BB K(\eps_1, \dots, \eps_p)$ with
$\Gamma^\circ$ in this orbit. By~\Cref{HypOctMinors}, the hyperoctahedral
action produces a principally regular matrix~$\Delta$ over the same field
such that $\Delta\!^\circ$ belongs to the hyperoctahedral orbit of the
identity as well. The~dependent sum in \Cref{EpsilonSum}, duality,
marginalization and conditioning and their reversals in \Cref{Embedding}
of algebraic Gaussians preserve realizability near a hyperoctahedral image
of the identity over their respective ground~sets. These facts and the
following \namecref{MainTool} are the main realizability tools for the
most complicated cases in the next section.

\begin{proposition} \label{MainTool}
If a gaussoid $\G G$ is (rationally) realizable near every one of
the $2^n$ hyperoctahedral images of the identity matrix, then every
hyperoctahedral image of $\G G$ is (rationally) near-identity
realizable, in particular positively realizable.
\end{proposition}

\begin{proof}
Let $\G H$ be in $\G G$'s hyperoctahedral orbit, arising from $\G G$
by a swap and a permutation. $\G H$~is realizable near the identity
if and only if $\G G$ is realizable near the matrix which is obtained
from the identity by permuting and swapping in reverse. These operations
result in a hyperoctahedral image of the identity near which $\G G$
is realizable by assumption. The hyperoctahedral action which transports
this curve of realizations back to realize $\G H$ near the identity
does not change the field, so rationality is preserved.
\end{proof}

\section{Proof of the main theorem}
\label{sec:Proof}

In this section a proof of our main result~\Cref{MainThm} is delivered in
the form of a series of lemmas which each settle one type of minimal
gaussoid extension of at most two CI statements. In each case we construct
a rational near-identity realization.

\begin{lemma} \label{AtmostOne}
All CI structures with at most one element are rationally realizable near
the identity.
\end{lemma}

\begin{proof}
The empty structure is realized by a symmetric matrix with $1$-diagonal
and independent variables in the off-diagonal entries. Clearly, none of
the almost-principal minors of this matrix vanish as polynomials.
Every singleton subset of $\A_N$ is vacuously a gaussoid. The singleton
gaussoids form a single orbit under the action of the hyperoctahedral group.
While this action does not in general preserve positive realizability,
we can emulate it using \Cref{Embedding} in a way that shows that it~\emph{is}
preserved in this case. Given any singleton~$(ij|K)$, first permute it to
$(12|K')$, marginalize to $12K'$ and then contract $K'$ to arrive at the
singleton~$(12|)$ over the ground set $N = 12$. These transformations
preserve rational positive realizability and their inverses do as well.
Thus we can transform every singleton into every other singleton while
preserving realizability and it remains to see that $\{ (12|) \}$ is
rationally positively realizable, which is trivial.
\end{proof}

From now on we consider two-element sets $\{ (ij|N), (kl|M) \}$ and their
minimal gaussoid extensions. Using the fact that marginalizations and
conditionings of Gaussians are Gaussians (over the same field) and that
we can undo these operations generically via \Cref{Embedding}, we can
assume that we work over the ground set $ijklNM$ and that $N \cap M =
\emptyset$. 

The gaussoid axioms have two antecedents. Every antecedent set of a gaussoid
axiom is therefore not a gaussoid. The following lemma deals with this type:

\begin{lemma}
If $\G B = \{ (ij|N), (kl|M) \}$ is not a gaussoid, then each of its
minimal gaussoid extensions has cardinality four and is rationally
realizable near the identity.
\end{lemma}

\begin{proof}
$\G B$ not being a gaussoid requires that $(ij|N)$ and $(kl|M)$ are
distinct and form the antecedent set of a gaussoid axiom. Thus the
two CI statements lie in a 3-face of the ambient $ijklNM$-cube;
see~\cite{ConstructionMethods}.
We can therefore reduce the study of gaussoid extensions of $\G B$
to this 3-face and hence, after conditioning, to a 3-element ground~set.
Every gaussoid closure of $\G B$ is thus a 3-gaussoid which is
placed in a 3-face of the $ijklNM$-cube. With two generators, each
closure has exactly four elements. The 3-gaussoids are all realizable
as undirected graphical models or their duals. Rational near-identity
realizations have been constructed in~\cite[Theorem~1]{LnenickaMatus}
and those are embedded back into the $ijklNM$-cube via~\Cref{Embedding}.
\end{proof}

The remaining type of gaussoids is comprised of pairs of so-called
\emph{inferenceless} generators with respect to the gaussoid axioms:
two-element subsets of $\A_N$ which are vacuously gaussoids.
We~expect this type to be the hardest to realize. The gaussoid axioms,
as the previous proof shows, govern inferences of two CI statements
in a common 3-face of the hypercube.
The realizabilities of inferenceless pairs prove that there are no
valid two-antecedental inference rules for Gaussian CI whose antecedents
lie further apart in the hypercube than in a common 3-face.

We continue to assume that the ground set is $ijklNM$ and that
$N \cap M = \emptyset$. In addition, the assumption of inferenceless
generators can be expressed as
\begin{align*}
\label{eq:Inf'less} \tag{$\dagger$}
|N \oplus M| \ge |ij \cap kl|.
\end{align*}
This type splits into a number of cases depending on how ``entangled''
$ij$, $kl$, $N$ and $M$ are, as these entanglements influence the form
of a potential realizing matrix. Up to the group $\BB Z/2 \times \Sym_N$
of duality and isomorphy, which preserves rational positive realizability,
and symmetries in the roles of $ij$ and $kl$, there are seven cases:

\begin{center}
\begin{tabular}{c||c|c|c|c|c|c|c}
& 1 & 2 & 3 & 4 & 5 & 6 & 7 \\ \hline
$ij \cap kl$ & $\emptyset$ & $\emptyset$ & $\emptyset$ & $\emptyset$ & $i$ & $i$ & $ij$ \\
$ij \cap M$  & $\emptyset$ & $i$ & $i$ & $ij$ & $\emptyset$ & $j$ & $\emptyset$ \\
$kl \cap N$  & $\emptyset$ & $\emptyset$ & $k$ & $\emptyset$ & $\emptyset$ & $\emptyset$ & $\emptyset$ \\
\end{tabular}
\end{center}

The hyperoctahedral group has a considerably wider reach. Every gaussoid
$\{ (ij|N), (kl|M) \}$ can be transformed into $\{ (ij|), (kl|M') \}$,
where $ij \cap M' = \emptyset$, by swapping out $N \cup (M \cap ij)$.
This action reduces the seven cases in the table above to only three cases:
\begin{itemize}
\item $\{ (ij|), (kl|M) \}$ on $ijklM$ with $ij \cap klM = \emptyset$,
\item $\{ (ij|), (ik|M) \}$ on $ijkM$ with $ij \cap kM = \emptyset$,
\item $\{ (ij|), (ij|M) \}$ on $ijM$ with $ij \cap M = \emptyset$.
\end{itemize}
Not only does $\HypOct_N$ decrease the number of cases, it also allows
to pick simpler representatives of each case. The three representatives
displayed above all contain the statement $(ij|)$, which only mandates
that a specific \emph{entry} of the realizing matrix be zero.
This reduction comes at the cost of not preserving positive
realizability. Using \Cref{MainTool}, to obtain rational positive
realizability of the entire orbit, we realize the above three case
representatives rationally near all the matrices which are equivalent
to the identity under the hyperoctahedral action.

The first case is the union of gaussoids $\{ (ij|) \}$ and $\{ (kl|M) \}$
over disjoint ground sets $ij$ and $klM$ and is already settled as an
instance of \Cref{EpsilonSum} with the rational near-identity
realizations constructed for singleton gaussoids in \Cref{AtmostOne}.
The remaining two cases are settled by \Cref{ijik,ijij} below.

\begin{lemma} \label{ijik}
The gaussoid $\{ (ij|), (ik|M) \}$ on the ground set $ijkM$ is rationally
realizable near all hyperoctahedral images of the identity.
\end{lemma}

\begin{proof}
We introduce the following notation:
\[
  \Phi = \kbordermatrix{
    &   i   &   j   &   k   &   M     \\
  i & \pm 1 &     0 &   \xi & \bm u^T \\
  j &     0 & \pm 1 &  \eta & \bm v^T \\
  k &   \xi &  \eta & \pm 1 & \bm w^T \\
  M & \bm u & \bm v & \bm w & \Sigma
  },
\]
where $(ij|)$ is already fulfilled by imposing the zero entry. The statement
$(ik|M)$ is equivalent to the following relation, after a Schur complement:
\[
  \xi = \bm u^T \Sigma^{-1} \, \bm w.
\]
Thus we impose this relation on $\xi$. All other appearing symbols are
supposed to be generic, i.e., $\eta$ is a variable, the vectors have
independent variable entries $u_m$, $v_m$, $w_m$, for $m \in M$, and
$\Sigma$ is a generic symmetric matrix with $\pm 1$-diagonal and independent
$\eps_{mn}$ off-diagonals. The signs of the diagonal elements of $\Phi$
are arbitrary but fixed. $\Phi$~is a matrix over $\BB Q(\eta, u_m, v_m, w_m,
\eps_{mn})$ whose off-diagonal entries tend to zero with the infinitesimal
variables and thus it approaches any hyperoctahedral image of the identity
matrix. The only denominator appears in $\xi$ and is the principal minor
$\det \Sigma$ with constant term $\pm 1$, which is infinitesimally non-zero.

By construction, $(ij|)$ and $(ik|M)$ hold for $\Phi$. It is clear that the
only interesting almost-principal minors are those involving $\xi$. For any
$N \subsetneq M$, the statement $(ik|N)$ surely does not hold because it is
equivalent~to
\[
  \bm u^T \Sigma^{-1} \, \bm w = \bm u_N^T \Sigma_N^{-1} \, \bm w_N,
\]
where the variables in $\bm u, \bm w, \Sigma$ are all independent.
There are four remaining cases: $(ik|jN)$, $(il|kN)$, $(kl|iN)$ and
$(lm|ikN)$, for relevant choices of $l, m$ and $N \subseteq M$.

The almost-principal minor $(ik|jN)$ is rewritten using Schur complement
with respect to~$jN$~to
\begin{align*}
  \det \Phi_{ik|jN} &= \det \Phi_{jN} \left(\xi -
    \begin{pmatrix} 0 & \bm u_N^T \end{pmatrix}
    \Phi_{jN}^{-1}
    \begin{pmatrix} \eta \\ \bm w_N^T \end{pmatrix}
  \right),
\end{align*}
which vanishes if and only if the parenthesized factor vanishes as a
rational function. Numerator and denominator of $\xi$ do not involve
the variable $\eta$, so it suffices to show that there is a monomial
divisible by $\eta$ with non-zero coefficient in the ``bilinear term''
inside the parentheses. All terms involving $\eta$ are
\vskip -2em 
\begin{align*}
  \eta \sum_{n \in N} u_n \left(\Phi_{jN}^{-1}\right)_{jn}.
\end{align*}
Each of these summands has a unique variable $u_n$ which does not appear
in $\Phi_{jN}$. If $N \not= \emptyset$, this ensures that the $\eta$
terms do not cancel and that the almost-principal minor does not vanish.
In case $N = \emptyset$, it is sufficient to remark that $\Phi_{ik|} =
\xi \not= 0$ because $M \not= \emptyset$ due to the assumption of
inferenceless generators~\eqref{eq:Inf'less}.

For the case $(il|kN)$ first assume that $l \not= j$. Laplace expansion on
the first row of the almost principal minor gives a sum
\[
  \det \begin{pmatrix}
    u_l & \xi & \bm u_N^T \\
    w_l & \pm 1 & \bm w_N \\
    \Sigma_{N,l} & \bm w_N^T & \Sigma_N
  \end{pmatrix}
  = u_l \det \Phi_{kN} \mp \dots
\]
of which the omitted terms are not divisible by~$u_l$. Since the constant
term of $\det \Phi_{kN}$ is $\pm 1$, the monomial $u_l$ arises in the sum
and cannot be canceled by other terms.
If $l = j$, then $(ij|kN)$ is equivalent~to
\[
  0 = \begin{pmatrix} \xi & \bm u_N^T \end{pmatrix}
  \Phi_{kN}^{-1}
  \begin{pmatrix} \eta \\ \bm v_N^T \end{pmatrix}.
\]
Again we investigate the terms divisible by $\eta$:
\[
  \eta \left(
    \xi (\Phi_{kN}^{-1})_{kk}
    + \sum_{n \in N} u_n (\Phi_{kN}^{-1})_{kn}
  \right).
\]
Since $\xi \not= 0$ and $(\Phi_{kN}^{-1})_{kk}$ has constant term $\pm 1$,
we find the monomial $\eta u_m w_m$ for some $m \in M$ in the numerator
of this almost-principal minor.

The case $(kl|iN)$ for $l \not= j$ is completely analogous to the previous
$(il|kN)$ one. In fact, the involved matrices are identical up to exchanging
the places of the generic vectors $\bm u$ and $\bm w$, which already play
symmetric roles in the definition of $\xi$. The matrix for $(kj|iN)$ is
\[
  \begin{pmatrix}
    \eta & \xi & \bm w_N^T \\
    0 & \pm 1 & \bm u_N^T \\
    \bm v_N & \bm u_N & \Sigma_N
  \end{pmatrix}
\]
and again Laplace expansion can be used to see that $\eta$ survives as
a monomial of degree one.

The last case is $(lm|ikN)$. If $j \not\in lm$, the almost-principal minor
of
\[
  \begin{pmatrix}
    \eps_{lm} & u_l & w_l & \Sigma_{l,N} \\
    u_m & \pm 1 & \xi & \bm u_N^T \\
    w_m & \xi & \pm 1 & \bm w_N^T \\
    \Sigma_{N,m} & \bm u_N & \bm w_N & \Sigma_N
  \end{pmatrix}
\]
has a monomial $\eps_{lm}$ via Laplace expansion in the first row.
The numerators of other summands in this expansion are not divisible
by~$\eps_{lm}$, making it impossible to cancel this degree-1 monomial.
Otherwise, without loss of generality, $m = j$ and the matrix
\[
  \begin{pmatrix}
    v_l & u_l & w_l & \Sigma_{l,N} \\
    0 & \pm 1 & \xi & \bm u_N^T \\
    \eta & \xi & \pm 1 & \bm w_N^T \\
    \bm v_N & \bm u_N & \bm w_N & \Sigma_N
  \end{pmatrix}
\]
is susceptible to the same Laplace expansion proof yielding a monomial~$v_l$.
\end{proof}

\begin{lemma} \label{ijij}
The gaussoid $\{ (ij|), (ij|M) \}$ on the ground set $ijM$ is rationally
realizable near all hyperoctahedral images of the identity.
\end{lemma}

\begin{proof}
We use the matrix pattern
\[
  \Phi = \kbordermatrix{
     &   i   &   j   &   M     \\
   i & \pm 1 &     0 & \bm u^T \\
   j &     0 & \pm 1 & \bm v^T \\
   M & \bm u & \bm v & \Sigma
  }
\]
with column vectors $\bm u$ and $\bm v$ and a generic matrix $\Sigma$ with
$\pm 1$-diagonal and independent $\eps_{mn}$ off-diagonals. Again, $(ij|)$
is imposed already by a zero entry. Unlike the situation of \Cref{ijik}, we
cannot make $(ij|M)$ hold by imposing a relation on the $ij$-entry of $\Phi$
which is already set to~zero. The equation for $(ij|M)$ is equivalent~to
\[
  0 = \bm u^T \Adj(\Sigma) \, \bm v,
\]
that is, $\bm u$ and $\bm v$ are orthogonal with respect to the (infinitesimally
principally regular) adjoint of~$\Sigma$. Equivalently we could have used
$\Sigma^{-1}$ but prefer not to introduce denominators into $\Phi$ needlessly.
To~enforce this relation, we define $\bm u$ and $\bm v$ via the Gram-Schmidt
process on vectors $\bm x$ and $\bm y$ of mutually independent variables indexed
by~$M$, as follows:
\begin{align*}
  u_k &= x_k, \\
  v_k &= \alpha_M y_k - \beta_M x_k,
\end{align*}
with the bilinear forms $\alpha_L = \bm x_L^T \Adj(\Sigma_L) \bm x_L$ and
$\beta_L = \bm x_L^T \Adj(\Sigma_L) \bm y_L$ for any $L \subseteq M$. This
completes the definition of $\Phi$, which is a matrix over $\BB Q(x_m, y_m,
\eps_{mn})$ whose off-diagonal entries tend to zero with the infinitesimal
variables, and clearly $\CIS{\Phi}$ contains $(ij|)$ and $(ij|M)$.
Evidently $(kl|N) \not\in \CIS{\Phi}$ whenever $j \not\in klN$ because
the almost-principal submatrix is generic in this case. The remainder of
the proof treats CI statements of the forms $(ij|N)$, $(jk|N)$ and $(kl|jN)$
each for all suitable $k$, $l$ and~$N \subseteq M$.

If $N$ is any non-empty subset of $M$, the almost-principal minor $(ij|N)$
becomes
\begin{align*}
  \det \Phi_{ij|N} &= \bm u_N^T \Adj(\Sigma_N) \, \bm v_N \\
  &= \alpha_M \bm x_N^T \Adj(\Sigma_N) \, \bm y_N - \beta_M \bm x_N^T \Adj(\Sigma_N) \, \bm x_N \\
  &= \alpha_M \beta_N - \alpha_N \beta_M.
\end{align*}
If $N \not= M$, it suffices to find a monomial in $\alpha_M \beta_N$ which
does not appear in $\alpha_N \beta_M$. Given $k \in N$ and $m \in M \setminus N$,
and using that $x_m^2$ only appears in $\alpha_M$ via the constant term~$\pm 1$
in the cofactor $(\Adj \Sigma)_{mm}$, such a monomial is $x_m^2 x_k y_k$.

Next consider type $(jk|N)$ with
\[
  \det \Phi_{jk|N} = \det \Phi_N \, (\alpha_M y_k - \beta_M x_k)
  - \Phi_{j,N} \Adj(\Phi_N) \, \Phi_{N,k}.
\]
By the assumption of inferenceless generators~\eqref{eq:Inf'less}, $|M| \ge 2$,
so there exists $m \in M \setminus k$. The expansion of $\alpha_M y_k$ produces
the monomial $x_m^2 y_k$ which does not appear in $\beta_M x_k$. Thus this monomial
arises from the product term. The remaining term is a bilinear form with respect
to~$\Adj(\Phi_N)$. Expanding the $\Phi_{j,N}$ vector, pretending that $x_i = y_i = 0$
in case $N \ni i$, we find
\[
  \Phi_{j,N} \Adj(\Phi_N) \, \Phi_{N,k}
  = \alpha_M \bm y_N^T \Adj(\Phi_N) \, \Phi_{k,N} - \beta_M \bm x_N^T \Adj(\Phi_N) \, \Phi_{k,N}.
\]
Each monomial in $\alpha_M$ or $\beta_M$ has total degree at least~$2$;
$y_n$, $x_n$ and $\Phi_{kn}$ are variables or zero if $n = i \in N$.
Under no circumstance does any monomial of total degree $3$ arise. This
proves that $x_m^2 y_k$ is a monomial with non-zero coefficient in the
expansion of $\det \Phi_{jk|N}$, hence $(jk|N) \not\in \CIS{\Phi}$.

The last type $(kl|jN)$ splits into two cases, depending on whether $i \in kl$
or not. The proofs are similar, so suppose first that $i \not\in kl$.
Then
\[
  \det \Phi_{kl|jN} = \det \Phi_{jN} \eps_{kl}
  - \Phi_{k,jN} \Adj(\Phi_{jN}) \, \Phi_{jN,l}.
\]
Because $\det \Phi_{jN}$ has constant term $\pm 1$, the monomial $\eps_{kl}$
appears in the above expansion of the almost-principal minor. It suffices
to show that the bilinear form term does not produce this monomial. Indeed,
\[
  \Phi_{k,jN} \Adj(\Phi_{jN}) \, \Phi_{jN,l}
  = \sum_{a,b \in jN} \Phi_{ak} \Phi_{bl} (\Adj \Phi_{jN})_{ab}
\]
sums over products of three polynomials of which at most the entry of the
adjoint may have a non-zero constant term. Analogously to above, this sum
has no monomial of total degree~$1$, so the $\eps_{kl}$ monomial cannot be
canceled. Finally,~if $l = i$, the $\eps_{kl}$ term in the calculation
above becomes $x_k$ instead and one of the coordinates in the bilinear
form term is the $0$ in $\Phi_{ij|}$. However, this does not interfere
with the argument.
\end{proof}

\section{Remarks and examples}
\label{sec:Remarks}

We have shown that on every ground set $N$, a Boolean formula in inference
form $\varphi: \bigwedge \G L \Rightarrow \bigvee \G M$ in variables $\A_N$
and with $|\G L| \le 2$ is valid for all regular Gaussian distributions if
and only if the gaussoid axioms on~$N$ logically imply~$\varphi$. The main
work of the proof consists of realizing all gaussoids with two elements, as
their realizability implies that no inference forms with two antecedents,
beyond the deductive closure of the gaussoid axioms, are valid for all
regular Gaussians. Since we constructed positive-definite rational realizations
(in every neighborhood of the identity matrix), which is the most restrictive
constellation among algebraic and positive realizations over characteristic~zero,
we obtain \Cref{TwoAntecedental} which in character reaches slightly beyond the
probabilistic origin of gaussoids and into the field of synthetic geometry.

\Cref{TwoAntecedental} does not hold in general in positive characteristic.
For example, it is easy to see that the only principally regular matrix
over $\GF(2)$ is the identity matrix, so the $\GF(2)$-algebraic Gaussians
satisfy many inference rules which are not implied by the gaussoid axioms.
The proof strategy of this paper begins to fail over finite fields in
\Cref{sec:Epsilon}. For example, \Cref{EpsilonSum} does not hold over~$\GF(2)$.

The result does not generalize to more antecedents either: a valid
\emph{three}-antecedental inference rule for Gaussians which is not implied
by the gaussoid axioms was found by Lněnička and Matúš in~\cite[Lemma~10,~(20)]{LnenickaMatus}.
The offending gaussoid is $\{ (12|3), (13|4), (14|2) \}$ over $N = 1234$.
It contains the antecedents of the instance in dimension 4 of Studený's
infinite list of independent inference rules~\cite{StudenyNonfinite}.
This family appears in the same function in infinite forbidden minor
proofs for semimatroids~\cite[Proposition~4]{MatusMinors} and (semidefinite)
Gaussian CI structures~\cite[Theorem~3.2]{SimecekNonfinite}.

In private correspondence, Milan Studený kindly pointed out that \Cref{TwoAntecedental}
is not a complete analogue to~\cite{StudenyTwoAntecedental} because it
concerns only CI inference forms over \emph{elementary} or \emph{local}
CI~statements $(ij|K) \in \A_N$ where $i$ and $j$ are singletons.
By contrast, statements of the form $(I,J|K)$ with pairwise disjoint
sets $I,J,K \subseteq N$ are \emph{global} CI statements.
It is an easy exercise to prove that for compositional graphoids the
following equivalence holds:
\[
  (I,J|K) \;\Leftrightarrow\; \bigwedge_{i\in I, j\in J} (ij|K).
\]
Therefore, a general theory of gaussoids can forgo global symbols.
Since local and global semigraphoids are in bijection, a global
semigraphoid can be recovered exactly from its intersection with
$\A_N$~\cite[Section~2]{MatusMinors}. However, for inference rules with
a bound on the number of antecedents, there is a major difference in
allowing global statements. For example, the single global CI statement
$(12,34|5)$ corresponds to the local statements $\{(13|5), (14|5), (23|5), (24|5)\}$
which have a unique minimal gaussoid extension with 16~elements.
This gaussoid is realizable, hence $(12,34|5)$ is not the antecedent set
of a non-trivial valid global inference rule for Gaussians, but this
is not covered by our proof. Hence we have

\begin{conjecture}
\label{GlobalTwoAntecedental}
All minimal gaussoid extensions of at most two global CI statements are
realizable.
\end{conjecture}

In 2004, in an effort to classify semigraphoids by closures of families of
``easy semigraphoids'' under certain operations, Matúš~\cite{MatusClassification}
proved a significant refinement of Studený's two-antecedental completeness:
his Theorem~2 states that all minimal semigraphoid extensions of two global
CI~statements are multilinear over every field, which implies discrete
realizability and hence, by \Cref{ApproxLemma}, two-antecedental completeness
of the semigraphoid axioms for discrete CI and also for multilinearity over
every field.
Following~\cite{MatusMatroids}, a semigraphoid is a \emph{semimatroid} if
it is the CI structure of a polymatroid rank function. It is \emph{multilinear}
(depending on context this is also called just \emph{linear}) if the rank
function is given by a subspace arrangement, cf.~\cite[Section~6.1]{MatusMinors}.
Multilinearity over every field of all semigraphoids with at most two
generators is a strong result which may be helpful in approaching
\Cref{GlobalTwoAntecedental}. However, there is no general relation
between multilinear semimatroids which happen to be gaussoids and
algebraic or positive realizability, as the following example shows.

\begin{example}[A multilinear but non-complex gaussoid]
\label{MultilinearNonalgebraic}
To find a 5-gaussoid which is not algebraic over the complex numbers,
one iterates through the $\HypOct_5$-orbit representatives of 5-gaussoids
computed in~\cite{Geometry}.
\Cref{ImplicationAlgebraic} describes a complex realizability test for
gaussoids in geometric terms, which is readily translated into the
language of commutative algebra, to be executed by a computer algebra
system like \texttt{Macaulay2}~\cite{M2}. We refer to
\cite[Section~4.4]{CoxLittleOShea} for further clarification of the
machinery of computer algebra.
However, the exact test from \Cref{ImplicationAlgebraic} requires
computing the following quotient of the radical ideal defined by
the gaussoid-under-test $\G M$:
\[
  \sqrt{\left\langle \det \Gamma_{ij|K} : (ij|K) \in \G M\right\rangle} :
  \left(
    \prod_{K \subseteq N} \det \Gamma_K \cdot
    \prod_{(ij|K) \not\in \G M} \det \Gamma_{ij|K}
  \right),
\]
where $\Gamma$ is a generic symmetric $5 \times 5$ matrix over $\BB C$.
The product on the right is a very large polynomial and computing the
quotient is often infeasible due to the amount of memory and CPU~time
required.
Instead of taking the ideal quotient by a product of minors, one can
successively take the quotient by each factor. This produces a
potentially \emph{larger} variety than the realization space of $\G M$
described in \Cref{ImplicationAlgebraic}, but the computation is feasible
and if the larger variety turns out to be empty, the gaussoid is certain
to be non-realizable. For some gaussoids, this method proves
non-realizability within a few~seconds.

The~multilinearity test is based on the characterization of the cone
spanned by multilinear polymatroids on five variables achieved by
Dougherty-Freiling-Zeger in \cite{LinearRank}. We thank Kenneth Zeger
for pointing us to the new location of the data mentioned in their paper,
\url{http://code.ucsd.edu/zeger/linrank}. Using the list of extreme rays
of the multilinear polymatroid cone, one can compute their CI structures
and check for every incoming non-algebraic gaussoid from the preceding
test whether it is equal to the intersection of all extreme structures
above it.

One gaussoid passing both tests~is
\[
  \G M = \{ (14|23), (14|35), (15|2), (15|4), (23|145), (24|5), (25|13), (34|1), (35|) \}.
\]
To confirm that this gaussoid is non-realizable over $\BB C$, it suffices
to write down the ideal defined by the almost-principal minors in~$\G M$,
compute its radical and saturate it at the product of the non-vanishing
\emph{entries} of the matrix as specified by~$\G M$, which is a fast
operation. The variety defined by the resulting ideal is a superset of
the algebraic realization space of $\G M$ and one can check that the
almost-principal minor $(34|125) \not\in \G M$ vanishes on this variety.
This proves the nine-antecedental inference rule
\[
  \bigwedge \G M \;\Rightarrow\; (34|125) \vee
    (12|) \vee (13|) \vee (14|) \vee (15|) \vee
    (23|) \vee (24|) \vee (25|) \vee
    (34|) \vee
    (45|)
\]
for algebraic Gaussians over $\BB C$. Since $\G M$ does not satisfy this
rule, it is not realizable over~$\BB C$.
To~prove multilinearity, one has to exhibit the face of the multilinear
polymatroid cone defined by $\G M$ and verify that a relatively interior
point realizes~$\G M$. This face is spanned by the following 37 extreme
rays. Each ray is written as an array of 31~single-digit integers in
``binary counter'' order, i.e., the ranks are listed in the order of
$1, 2, 12, 3, 13, 23, 123, 4, \dots$, omitting the empty set at the
beginning whose value is always~$0$:

{\tiny
\begin{align*}
  0 0 0 1 1 0 0 1 1 0 1 1 1 1 1 0 1 1 1 1 1 1 1 1 1 1 1 1 1 1 1 &&
  1 0 1 0 0 1 1 1 1 1 0 0 1 1 0 1 1 1 1 1 1 1 1 0 1 1 1 1 1 1 1 &&
  1 1 0 0 0 1 1 1 1 1 1 1 0 0 0 1 1 1 1 1 1 1 1 1 0 1 1 1 1 1 1 \\[-.2em]
  1 0 1 1 0 1 1 1 1 1 1 0 1 1 1 1 1 1 1 1 1 1 1 1 1 1 1 1 1 1 1 &&
  1 1 1 0 0 1 1 1 1 1 1 1 1 1 0 1 1 1 1 1 1 1 1 1 1 1 1 1 1 1 1 &&
  1 1 0 1 1 1 1 1 1 1 1 1 1 1 1 1 1 1 1 1 1 1 1 1 1 1 1 1 1 1 1 \\[-.2em]
  0 0 1 1 1 0 1 1 1 1 1 1 2 2 2 1 1 1 2 2 2 2 2 2 2 2 2 2 2 2 2 &&
  1 0 1 0 1 1 2 1 2 1 0 1 1 2 1 2 1 2 2 2 2 1 2 1 2 2 2 2 2 2 2 &&
  1 1 1 0 1 2 2 1 2 2 1 1 1 2 1 2 2 2 2 2 2 2 2 1 2 2 2 2 2 2 2 \\[-.2em]
  1 1 1 0 1 1 2 1 2 2 1 2 1 2 1 2 1 2 2 2 2 2 2 2 2 2 2 2 2 2 2 &&
  1 1 1 1 0 2 1 2 1 2 2 1 2 1 1 2 2 2 2 1 2 2 2 2 2 2 2 2 2 2 2 &&
  1 2 1 0 1 2 2 1 2 2 2 2 1 2 1 2 2 2 2 2 2 2 2 2 2 2 2 2 2 2 2 \\[-.2em]
  1 1 1 1 1 2 2 1 2 2 2 1 2 2 2 2 2 2 2 2 2 2 2 2 2 2 2 2 2 2 2 &&
  1 1 0 1 1 2 1 2 2 1 2 2 1 1 2 2 3 3 2 2 3 2 2 3 2 3 3 3 3 3 3 &&
  1 1 1 1 1 2 1 2 2 2 2 2 2 2 2 2 3 3 2 2 3 3 3 3 3 3 3 3 3 3 3 \\[-.2em]
  1 2 1 1 1 3 2 2 2 3 3 2 2 2 2 3 3 3 3 2 3 3 3 3 3 3 3 3 3 3 3 &&
  1 1 1 1 1 2 2 2 2 2 2 2 2 2 2 3 3 3 3 2 3 2 3 3 3 3 3 3 3 3 3 &&
  1 1 1 1 1 2 2 2 2 2 2 2 2 2 2 2 3 3 3 3 3 3 3 3 2 3 3 3 3 3 3 \\[-.2em]
  1 1 2 1 1 2 2 2 2 2 2 2 3 3 2 2 3 3 3 3 3 3 3 3 3 3 3 3 3 3 3 &&
  2 1 2 1 1 2 3 2 3 3 2 2 3 3 2 3 2 3 3 3 3 3 3 3 3 3 3 3 3 3 3 &&
  1 1 1 1 2 2 2 2 3 2 2 3 2 3 3 3 3 4 3 3 4 3 4 4 4 4 4 4 4 4 4 \\[-.2em]
  1 2 1 1 2 3 2 2 3 3 3 3 2 3 3 3 4 4 3 3 4 4 4 4 4 4 4 4 4 4 4 &&
  1 1 2 2 1 2 2 3 2 3 3 2 4 3 3 3 4 3 4 3 4 4 4 4 4 4 4 4 4 4 4 &&
  1 1 2 1 2 2 3 2 3 2 2 3 3 4 3 3 3 4 4 4 4 3 4 4 4 4 4 4 4 4 4 \\[-.2em]
  1 1 2 1 2 2 3 2 3 3 2 3 3 4 3 3 3 4 4 4 4 4 4 4 4 4 4 4 4 4 4 &&
  2 2 2 1 1 3 3 3 3 4 3 3 3 3 2 4 4 4 4 3 4 4 4 4 4 4 4 4 4 4 4 &&
  1 2 2 1 2 3 3 2 3 3 3 3 3 4 3 3 4 4 4 4 4 4 4 4 4 4 4 4 4 4 4 \\[-.2em]
  1 2 2 2 2 3 3 3 3 4 4 3 4 4 4 4 5 4 5 4 5 5 5 5 5 5 5 5 5 5 5 &&
  2 1 3 2 2 3 4 3 4 4 3 3 5 5 4 4 4 5 5 5 5 5 5 5 5 5 5 5 5 5 5 &&
  2 3 2 1 2 4 4 3 4 5 4 4 3 4 3 5 5 5 5 4 5 5 5 5 5 5 5 5 5 5 5 \\[-.2em]
  3 2 3 2 1 4 4 4 4 5 4 3 5 4 3 5 5 5 5 4 5 5 5 5 5 5 5 5 5 5 5 &&
  2 2 2 2 2 4 4 4 4 4 4 4 4 4 4 5 5 6 6 5 6 5 6 6 5 6 6 6 6 6 6 &&
  2 2 3 2 2 4 4 4 4 5 4 4 5 5 4 5 6 6 6 5 6 6 6 6 6 6 6 6 6 6 6 \\[-.2em]
  2 1 3 2 3 3 5 3 5 4 3 4 5 6 5 5 4 6 6 6 6 5 6 6 6 6 6 6 6 6 6 &&
  2 2 3 2 3 4 5 4 5 5 4 5 5 6 5 6 5 7 7 6 7 6 7 7 7 7 7 7 7 7 7 &&
  2 3 3 2 3 5 5 4 5 6 5 5 5 6 5 6 7 7 7 6 7 7 7 7 7 7 7 7 7 7 7 \\[-.2em]
  3 2 4 3 2 5 5 5 5 6 5 4 7 6 5 6 7 7 7 6 7 7 7 7 7 7 7 7 7 7 7
\end{align*}
}
The fact that these vectors are multilinear and even extreme rays is the
content of~\cite{LinearRank}. It~is easy to verify that the CI structure
of the sum of the above vectors is exactly equal to $\G M$.
In their published data, \cite{LinearRank}~give integer matrices whose
rowspans represent the subspace arrangement. These matrices have the
additional property that whenever a concatenation of them has rank $r$
over $\BB Q$, then there exists an $r \times r$ submatrix whose determinant
equals~$\pm 1$, thus proving that the matrices realize the same polymatroid
over every field. By the general theory of multilinear semimatroids,
this shows that $\G M$ is multilinear over every field as~well.
The process and code snippets are more thoroughly documented at
\url{https://github.com/taboege/gaussant-code}.
\end{example}

\nocite{MatusII}
Putting together results by Matúš-Studený~\cite{MatusStudeny}--\cite{MatusFinal}
(corrections by Šimeček~\cite{SimecekShortnote}) and Lněnička-Matúš~\cite{LnenickaMatus},
one finds that on the set of gaussoids on a four-element ground set,
the realizability by discrete random variables and by positive-definite
matrices over $\BB Q$ or $\BB R$ are equivalent. In particular, a
ground set of cardinality five is required to find a multilinear but
non-algebraic gaussoid.

The notion of realizability near a matrix which is in the hyperoctahedral
orbit of the identity matrix was crucial to the proof of the main theorem
in \Cref{sec:Proof}. We showed that all gaussoids generated from at most
two CI statements are rationally realizable near the identity matrix.
In the following we present counterexamples to some obvious questions
about near-identity realizability.

\begin{example}[A non-near-identity realizable gaussoid]
Entry №~20~in~\cite[Table~1]{LnenickaMatus} contains the curve of matrices
\[
  \begin{psmallmatrix}
    1 & 2-\delta^{-2} & \delta & \delta \\
    2-\delta^{-2} & 1 & 0 & \delta \\
    \delta & 0 & 1 & \delta^2 \\
    \delta & \delta & \delta^2 & 1
  \end{psmallmatrix}
  \quad
  \xrightarrow\quad
  \quad
  \begin{psmallmatrix}
    1 & \frac{2}{9} & \frac{3}{4} & \frac{3}{4} \\
    \frac{2}{9} & 1 & 0 & \frac{3}{4} \\
    \frac{3}{4} & 0 & 1 & \frac{9}{16} \\
    \frac{3}{4} & \frac{3}{4} & \frac{9}{16} & 1 \\
  \end{psmallmatrix},
  \;
  \text{as $\delta \to \frac{3}{4}$}.
\]
The gaussoid $\G G = \{ (13|24), (23|), (34|1) \}$ realized by this matrix
is the algebraic Gaussian in our sense over $\BB Q(\eps)$ where
$\frac{3}{4} + \eps$ is substituted for~$\delta$.

Consider the slice of the positive realization space over $\BB R$ of this
gaussoid on the affine-linear space of symmetric matrices
\[
  \Sigma = \begin{psmallmatrix}
  1 & a & b & c \\
  a & 1 & 0 & e \\
  b & 0 & 1 & f \\
  c & e & f & 1
  \end{psmallmatrix}.
\]
This slice is the intersection of the algebraic realization space of the
gaussoid with the elliptope. The~identity matrix lies in the center of
the elliptope and we wish to show that the realization space of $\G G$,
although non-empty, does not approach this center. In particular $\G G$
is a gaussoid with \emph{three} elements, rationally positively realizable,
but not realizable near the identity. On the given slice, these two
equations hold:
\begin{align*}
f &= bc, \tag{$34|1$} \\
b + aef &= cf + be^2. \tag{$13|24$}
\end{align*}
Substituting the first into the second equality and canceling the non-zero
factor~$b$ in every term we find
\[
  1 + ace = c^2 + e^2.
\]
This equation cannot be satisfied if $a$, $c$ and $e$ all tend to zero.
It can be shown that the realization space of $\G G$ decomposes into eight
\emph{reorientation classes} which are identical up to an orthogonal
transformation; for an explanation of reorientation, see~\cite[Section~5]{Geometry}.
This transformation preserves Euclidean distances and it fixes the center
of the elliptope, thus all of these components have the same distance to
the identity matrix. Focusing on one of them, we can assume that $a$, $c$
and $e$ are all positive and then the Euclidean distance of a realization
of~$\G G$ to the identity is
\[
  \sqrt{2} \sqrt{a^2 + b^2 + c^2 + e^2 + f^2} = \sqrt{2} \sqrt{1 + ace + a^2 + b^2 + f^2} \ge \sqrt{2}.
\]
By allowing $e$ to converge to one while $a$, $b$ and $c$ converge to zero,
one can find positive-definite realizations of $\G G$ which approach this
lower bound. Thus, the elliptope slice of the realization space of $\G G$
has distance $\sqrt{2}$ to the identity matrix.
\end{example}

It was remarked in~\cite[Corollary~1]{Geometry} that, based on the hyperoctahedral
action and~\cite[Table~1]{LnenickaMatus}, every $4$-gaussoid is algebraically
realizable over~$\BB C$. Since all realizations in the table are even rational,
our~\Cref{RationalFunctions} and \Cref{HypOctMinors} furthermore imply that every
$4$-gaussoid is algebraically realizable over~$\BB Q$, while not all of them
are positively realizable even over~$\BB R$.

\begin{example}[Near-identity realizability not preserved under \unboldmath$\HypOct_N$]
\label{NearHypOct}
The bracketed self-dual gaussoid $\{ (12|), (12|34), (34|1), (34|2) \}$ in
item~$4_{12}$ in~\cite[p.~15]{Geometry} is not positively realizable over~$\BB R$.
However, in its hyperoctahedral orbit is the likewise self-dual
$\{ (12|3), (12|4), (34|1), (34|2) \}$ and this gaussoid is even realizable
rationally near the identity matrix --- it is №~30~in~\cite[Table~1]{LnenickaMatus}.
\end{example}

\begin{example}[A non-algebraic gaussoid] \label{Nonalgebraic}
In~\cite[Example~13]{Geometry} a $5$-gaussoid was found which is not algebraically
realizable over~$\BB C$. This gaussoid had a redundant element $(25|34)$ which
contributed neither to the property of being a gaussoid nor to being non-realizable.
In this example, we consider the gaussoid of \cite[Example~13]{Geometry},
with~$(25|34)$~removed, from the broader perspective of algebraic realizability over
general fields. The claim is that
\[
  \G V = \{ (12|),
    (13|4), (14|5), (23|5), (35|1), (45|2),
    (15|23), (34|12), 
    (24|135)
  \}
\]
is not algebraically realizable over any field. It is sufficient to check
this over algebraically closed fields. Over these fields, we can impose a
unit diagonal on a principally regular realization~$\Gamma$, by \Cref{RealizDiagonals}.
Then, $\G V$~imposes the following easy equations on~$\Gamma$:
\begin{equation*}
\begin{aligned}[c]
  \Gamma = \begin{psmallmatrix}
    1 & 0 & b & c & d \\
    0 & 1 & e & f & g \\
    b & e & 1 & h & i \\
    c & f & h & 1 & j \\
    d & g & i & j & 1
  \end{psmallmatrix},
\end{aligned}
\qquad
\begin{aligned}[c]
  e &= gi,
\end{aligned}
\qquad
\begin{aligned}[c]
  i &= bd, \\
  b &= ch,
\end{aligned}
\qquad
\begin{aligned}[c]
  c &= dj, \\
  j &= fg.
\end{aligned}
\end{equation*}
Using these variable substitutions, the longer equations, corresponding to
CI statements with bigger conditioning sets, shrink. They are
\begin{align*}
  0 &= d[1 - d^2 f^2 g^2 h^2 (1 + d^2 g^2 - g^2)], \tag{$15|23$} \\[-.2em]
  0 &= h[1 - 2 d^2 f^2 g^2], \tag{$34|12$} \\[-.2em]
  0 &= f[(1 - d^2) (1 + d^2 f^2 g^4 h^2 (1 + d^2) - g^2 (1 + d^2 h^2 (1 + f^2)))] \tag{$24|135$}.
\end{align*}
Dividing by the non-zero variables $d$, $f$, and $h$ yields a system in
even powers of the variables. Replacing $d^2 = D$ and so on, we have:
\begin{align*}
  1 &= DFGH(1 + DG - G), \tag{a} \label{eq:a} \\[-.2em]
  1 &= 2DFG, \tag{b} \label{eq:b} \\[-.2em]
  0 &= (1 - D) (1 + DFG^2H + D^2FG^2H - G - DGH - DFGH) \tag{c} \label{eq:c}.
\end{align*}
In a principally regular $\Gamma$ we have $1 - D = \det \Gamma_{15} \not= 0$,
so the last equation is equivalent to
\begin{align*}
  0 = 1 + DFG^2H + D^2FG^2H - G - DGH - DFGH. \tag{c$'$} \label{eq:c'}
\end{align*}
By adding up \eqref{eq:a} and \eqref{eq:c'} and then using \eqref{eq:b}:
\begin{align*}
  0 &= 1 + DFG^2H + D^2FG^2H - G - DGH - DFGH + DFGH + D^2FG^2H - DFG^2H - 1 \\[-.2em]
    &= 2D^2FG^2H - DGH - G \\[-.2em]
    &= -G,
\end{align*}
which is a contradiction to $(25|) \not\in \G V$. Notice that this
contradiction was derived using only the non-vanishing of principal
minors and division by variables known to be non-zero by the definition
of~$\G V$. The~argument is independent of the field chosen, so the
following nine-antecedental inference rule is valid for algebraic
Gaussians over every field:
\[
  \label{VamosInf} \tag{$\ddagger$}
  \bigwedge \G V \;\Rightarrow\; (15|) \vee (24|) \vee (25|) \vee (34|).
\]
This proves that~$\G V$ is not algebraically realizable over any field.

We note that, unlike $\G M$ of \Cref{MultilinearNonalgebraic}, this CI~structure
is not realizable by discrete random variables either because it is not a
semimatroid. The closure of $\G V$ in the set of $5$-semimatroids and in the
set of multilinear $5$-semimatroids can be computed using polyhedral geometry,
as outlined in~\Cref{MultilinearNonalgebraic}, because the extreme rays of
the relevant cones are known.
The semimatroid closure of $\G V$ contains 16 CI~statements but none of the
conclusions $(15|)$, $(24|)$, $(25|)$ and $(34|)$ of~\eqref{VamosInf}, so
this inference rule is not valid for semimatroids. The multilinear closure
is larger at 22 elements and from the possible conclusions of \eqref{VamosInf},
it contains precisely $(15|)$. Hence, the special case $\bigwedge \G V
\Rightarrow (15|)$ of \eqref{VamosInf} is valid for multilinear semimatroids.
The closure of~$\G V$ among the CI~structures realizable by discrete random
variables lies in between the semimatroid and multilinear closures.
It remains open whether $\bigwedge \G V \Rightarrow (15|)$ is valid
even in the discrete case.
\end{example}

As a final direction for future work we would like to mention the influence
of the field $\BB K$ on the realizability of gaussoids. There are numerous
questions to be asked inspired by related theorems in matroid theory,
cf.~\cite[Chapter~6]{Oxley}. To give one example: recall from \Cref{ImplicationAlgebraic}
that the realizability of a gaussoid $\G G$ is a disproof of the existence
of a non-trivial valid inference rule with antecedent set~$\G G$.
Therefore, realizing matrices represent invalidity proofs for inference
rules in a very compact way.
In~the proof of \Cref{MainThm}, we exhibited \emph{rational} matrices as
invalidity proofs for all two-antecedental inference formulas which are
not implied by the gaussoid axioms. These matrices can be stored on a
computer and verified using exact arbitrary-precision rational arithmetic.
This leads to the following two variants of a question which was asked
before by Petr Šimeček~\cite[Section~4]{SimecekGaussian} in the broader
context of semidefinite Gaussian CI~models. He asked whether there exists
a realizable Gaussian CI~structure for which a \emph{rational} realizing
covariance matrix does not exist --- because his computer search for
invalid inference rules was conducted on rational matrices only.
Both questions are answered affirmatively in~\cite{Gaussruler}.

\begin{question}
\label{PositiveRational}
Is there a gaussoid which is positively realizable over $\BB R$ but not
over~$\BB Q$?
\end{question}

\begin{question}
\label{AlgebraicDiff}
Are there gaussoids which separate the notions of algebraic realizability
over $\BB C$, $\BB R$ and~$\BB Q$?
\end{question}

\paragraph{\bfseries Acknowledgement}
The author wishes to thank Thomas Kahle, Andreas Kretschmer and Milan Studený
for helpful discussions, and the anonymous referees for their clear and
thoughtful suggestions for improvement of the earlier manuscript.
Thanks to Xiangying Chen for spotting an error in the previous version of
\Cref{Nonalgebraic}. This work is funded by the Deutsche Forschungsgemeinschaft
(DFG, German Research Foundation) -- 314838170, GRK 2297 MathCoRe.

\bibliographystyle{alpha}
\bibliography{gaussant}

\end{document}